\documentclass[12pt,reqno]{amsart}

\usepackage{amssymb,hyperref}

\newtheorem{theorem}{Theorem}[section]
\newtheorem{proposition}[theorem]{Proposition}
\newtheorem{lemma}[theorem]{Lemma}
\newtheorem{corollary}[theorem]{Corollary}

\theoremstyle{definition}
\newtheorem{definition}[theorem]{Definition}
\newtheorem{remark}[theorem]{Remark}

\numberwithin{equation}{section}

\begin{document}

\baselineskip=15pt

\title[Branched projective structures and branched ${\rm SO}(3,{\mathbb C})$-opers]{Branched
projective structures, branched ${\rm SO}(3,{\mathbb C})$-opers and
logarithmic connections on jet bundle}

\author[I. Biswas]{Indranil Biswas}

\address{School of Mathematics, Tata Institute of Fundamental
Research, Homi Bhabha Road, Mumbai 400005, India}

\email{indranil@math.tifr.res.in}

\author[S. Dumitrescu]{Sorin Dumitrescu}

\address{Universit\'e C\^ote d'Azur, CNRS, LJAD, France}

\email{dumitres@unice.fr}



\date{}

\begin{abstract} 
We study the branched holomorphic projective structures on a compact Riemann surface $X$ with a
fixed branching divisor $S\, =\, \sum_{i=1}^d x_i$, where $x_i \,\in\, X$ are distinct points.
After defining branched ${\rm SO}(3,{\mathbb C})$--opers, we show 
that the branched holomorphic projective structures on $X$ are in a natural bijection
with the branched ${\rm SO}(3,{\mathbb C})$--opers singular at $S$. It is deduced 
that the branched holomorphic projective structures on $X$ are also identified
with a subset of the space of all logarithmic connections
on $J^2((TX)\otimes {\mathcal O}_X(S))$ singular over $S$, 
satisfying certain natural geometric conditions.
\end{abstract}

\maketitle

\tableofcontents

\section{Introduction}\label{se1}

The role of the space of holomorphic projective structures in the understanding of the uniformization theorem of 
Riemann surfaces was emphasized by many authors (see, for instance, \cite{Gu,St} and references therein).
Recall that a holomorphic projective structure on a Riemann surface $X$ is given by (the equivalence class of) an 
atlas $\{(U_i,\, \phi_i)\}_{i\in I}$ with local charts $X\,\supset \, U_i\, \stackrel{\phi_i}{\longrightarrow}\, {\mathbb C}{\mathbb P}^1$, $i\, \in\, I$,
such that all the transition maps $\phi_j\circ \phi^{-1}_i$
are restrictions of elements in the M\"obius group $\text{PGL}(2,{\mathbb C})$ of complex projective 
transformations (see Section \ref{se2.1} for more details).

Historically, the idea of using holomorphic projective structures to prove uniformization theorem for Riemann 
surfaces came from the study of second-order linear differential equations (see, \cite[Chapter VIII]{St}). The modern 
point of view summarizes this equivalence between holomorphic projective connections and second-order linear 
differential equations as an identification of the space of holomorphic projective connections with the space of 
${\rm PGL}(2,{\mathbb C})$--opers (see \cite[Section I.5]{De} and \cite{BeDr}).

A more general notion is that of a branched holomorphic projective structure which was introduced and studied by 
Mandelbaum in \cite{Ma1, Ma2}; more recently, a general notion of (non necessarily flat) branched Cartan geometry on complex manifolds was 
introduced and studied in \cite{BD}. A branched holomorphic projective structure is defined by (the equivalence 
class of) a holomorphic atlas with the local charts being finite branched coverings of open subsets in ${\mathbb 
C}{\mathbb P}^1$, while the transition maps are restrictions of elements in the M\"obius group $\text{PGL}(2,{\mathbb 
C})$ (see Section \ref{se5.1} for details). The branching divisor $S$ of a branched projective structure on $X$ is 
the union of points in $X$ where the local projective charts admit a ramification point. The Riemann surface $X 
\setminus S$ inherits a holomorphic projective structure in the classical sense. Branched holomorphic projective 
structures play an important role in the study of hyperbolic metrics with conical singularities on Riemann surfaces 
or in the study of codimension one transversally projective holomorphic foliations (see, for example, \cite{CDF}).

From a global geometric point of view, a holomorphic projective structure over a Riemann surface $X$ is known to
give a flat ${\mathbb C}{\mathbb P}^1$-bundle over $X$ together with a holomorphic section which is transverse to 
the horizontal distribution defining the flat connection on the bundle. For a branched holomorphic projective 
structure with branching divisor $S$, we have a similar description, but the section of
the projective bundle fails to be transverse 
to the horizontal distribution (defining the flat structure) precisely at points in $S$ (see, for example, \cite{BDG, 
CDF, GKM, LM} or Section \ref{se5.1} here).

In this article we study the branched holomorphic projective structures on a Riemann surface $X$ with fixed 
branching divisor $S$. Throughout, we assume that $S$ is reduced, meaning that $S\, :=\, \sum_{i=1}^d x_i$, 
with $x_i \,\in\, X$, distinct points.

The presentation is organized in the following way. Sections \ref{se2}, \ref{se3} and \ref{sec4} deal with the 
geometry of classical holomorphic projective connections and Sections \ref{se5}, \ref{sec6} and \ref{sec7} are about 
branched holomorphic projective connections. More precisely, Section \ref{se2} presents the geometrical setting of a 
holomorphic projective structure on a Riemann surface $X$ and proves that
a projective structure induces a holomorphic connection on the 
rank three holomorphic 2-jet bundle $J^2(TX)$. Moreover this connection has a natural geometric behavior: it is an 
oper connection (see Definition \ref{def1}) with respect to the canonical filtration of $J^2(TX)$ induced by the 
kernels of the natural forgetful projections of jet bundles $J^2(TX)\,\longrightarrow\,
J^1(TX)$ and $J^2(TX)\,\longrightarrow\, TX$. In Section 
\ref{se3} we identify the space of holomorphic projective structures on the Riemann surface $X$ with the space of 
oper connections on $J^2(TX)$ satisfying some extra natural geometrical properties (see Corollary \ref{cor1}). 
Section \ref{sec4} translates the previous equivalence as an identification of the space of holomorphic projective 
connections with the space of ${\rm SO}(3,{\mathbb C})$--opers.

Section \ref{se5} starts the study of branched holomorphic projective structures. It is shown that a branched 
projective structure on a Riemann surface $X$ with branching divisor $S$ gives rise to a logarithmic connection on 
the rank two 2-jet bundle $J^2((TX)\otimes {\mathcal O}_X(S))$ singular over $S$, with residues at $S$ satisfying 
certain natural geometric conditions with respect to the canonical filtration of $J^2((TX)\otimes {\mathcal O}_X(S))$ 
(see Proposition \ref{prop3}).

The main results proved here are obtained in Section \ref{sec6} and in Section \ref{sec7}.

In Section \ref{se6.1} we
introduce the notion of a branched ${\rm SO}(3,{\mathbb C})$--oper singular at the divisor $S$. We 
show that the space ${\mathcal P}_S$ of branched holomorphic projective structures with fixed branching divisor $S$ 
is naturally identified with the space of branched ${\rm SO}(3,{\mathbb C})$--opers singular at $S$ (see Theorem 
\ref{thm1}).

We deduce that this space $\mathcal P_S$ also coincides with a subset of the set of all logarithmic connections with 
singular locus $S$, satisfying certain natural geometric conditions, on the rank three holomorphic $2$-jet bundle 
$J^2((TX)\otimes {\mathcal O}_X(S))$ (see Proposition \ref{prop4} and Theorem \ref{thm2}).

The above mentioned Theorem \ref{thm2} generalizes the main result in \cite{BDG} (Theorem 5.1) where, under the additional 
assumption that the degree $d$ of $S$ is even and such that $d \,\neq\, 2g-2$ (with $g$ the genus of $X$), it was 
proved that ${\mathcal P}_S$ coincides with a subset of the set of all logarithmic connections with singular locus 
$S$, satisfying certain geometric conditions, on the rank two holomorphic jet bundle $J^1(Q)$, where $Q$ is a fixed 
holomorphic line bundle on $X$ such that $Q^{\otimes 2}\,=\, TX\otimes {\mathcal O}_X(S)$. It may
be mentioned that for a branching divisor $S$ of general degree $d$, the bundle $Q$
considered in \cite{BDG} does not exist.

Let us clarify an important improvement in the methods developed in this work with respect to those in \cite{BDG}. 
The $\text{PSL}(2,{\mathbb C})$-monodromy of a (unbranched) holomorphic projective structure on a compact Riemann 
surface $X$ always admits a lift to $\text{SL}(2,{\mathbb C})$ (see \cite[Lemma 1.3.1]{GKM}). Geometrically 
this means that the associated flat ${\mathbb C}{\mathbb P}^1$-bundle over $X$ is the projectivization of a rank 
two vector bundle over $X$; hence one can work in the set-up of rank two vector bundles. This is not true anymore for 
branched holomorphic projective structures if the degree of the branching divisor is odd; for this reason the results 
and methods in \cite{BDG} cannot be extended to the case where the branching divisor of the branched holomorphic 
projective structure is of odd degree.

Here we consider ${\rm SO}(3,{\mathbb C})$--opers instead of
the equivalent $\text{PSL}(2,{\mathbb C})$-opers and we develop the 
notion of branched ${\rm SO}(3,{\mathbb C})$--opers. This enables us to investigate branched holomorphic 
projective structure in the framework of rank three holomorphic (2-jet) vector bundles instead of the
projective bundles.

\section{Projective structure and second jet of tangent bundle}\label{se2}

\subsection{Projective structure}\label{se2.1}

The multiplicative group of nonzero complex numbers will be denoted by ${\mathbb C}^*$.
Let $\mathbb V$ be a complex vector space of dimension two. Let ${\mathbb P}(\mathbb V)$ denote 
the projective line that parametrizes all one-dimensional subspaces of $\mathbb V$. Consider 
the projective linear group $\text{PGL}(\mathbb V) \, :=\, \text{GL}(\mathbb V)/({\mathbb C}^*\cdot{\rm
Id}_{\mathbb V})\,=\, \text{SL}(\mathbb V)/(\pm {\rm 
Id}_{\mathbb V})$. The action of $\text{GL}(\mathbb V)$ on $\mathbb V$ produces an action of 
$\text{PGL}(\mathbb V)$ on ${\mathbb P}(\mathbb V)$. This way $\text{PGL}(\mathbb V)$ gets 
identified with the group of all holomorphic automorphisms of ${\mathbb P}(\mathbb V)$.

Let $\mathbb X$ be a connected Riemann surface. A holomorphic coordinate function on $\mathbb X$
is a pair $(U,\, \phi)$, where $U\, \subset\, \mathbb X$ is an open subset and
$$
\phi\, :\, U\, \longrightarrow\, {\mathbb P}(\mathbb V)
$$
is a holomorphic embedding. A holomorphic coordinate atlas on $\mathbb X$ is a family of
holomorphic coordinate functions $\{(U_i,\, \phi_i)\}_{i\in I}$ such that
$\bigcup_{i\in I} U_i \,=\, \mathbb X$. So
$\phi_j\circ \phi^{-1}_i \, :\, \phi_i(U_i\cap U_j) \, \longrightarrow\, \phi_j(U_i\cap U_j)$
is a biholomorphic map for every $i,\, j\,\in\, I$ with $U_i\cap U_j\, \not=\, \emptyset$.

A projective structure on $\mathbb X$ is given by a holomorphic coordinate atlas
$\{(U_i,\, \phi_i)\}_{i\in I}$ such that for all ordered pairs $i,\, j\, \in\, I$,
with $U_i\cap U_j\, \not=\, \emptyset$, and every connected component $U\, \subset\, U_i\cap U_j$, 
there is an element $G^U_{j,i}\, \in\, \text{PGL}(\mathbb V)$
satisfying the condition that the biholomorphic map
$$
\phi_j\circ \phi^{-1}_i \, :\, \phi_i(U) \, \longrightarrow\, \phi_j(U)
$$
is the restriction, to $\phi_i(U)$, of the automorphism of ${\mathbb P}(\mathbb V)$ 
given by of $G^U_{j,i}$. Note that $G^U_{j,i}$ is uniquely determined by $\phi_j\circ \phi^{-1}_i$. 

Two holomorphic coordinate atlases $\{(U_i,\, \phi_i)\}_{i\in I}$ and $\{(U'_i,\, \phi'_i)\}_{i\in 
J}$ satisfying the above condition are called \textit{equivalent} if their union $\{(U_i,\, \phi_i)\}_{i\in I}
\bigcup \{(U'_i,\, \phi'_i)\}_{i \in J}$ also satisfies the above condition. A \textit{projective structure} on 
$\mathbb X$ is an equivalence class of atlases satisfying the above condition; see \cite{Gu}, 
\cite{He}, \cite{GKM} for projective structures.

Let $\gamma\, :\,{\mathbf P}\, \longrightarrow\, \mathbb X$ be a holomorphic ${\mathbb C}{\mathbb P}^1$--bundle
over $\mathbb X$. In other words, $\gamma$ is a surjective holomorphic submersion such that each fiber of it
is holomorphically isomorphic to the complex projective line ${\mathbb C}{\mathbb P}^1$. Let
\begin{equation}\label{tga}
T_\gamma\, :=\, \text{kernel}(d\gamma)\, \subset\,T{\mathbf P}
\end{equation}
be the relative holomorphic tangent bundle, where $T{\mathbf P}$ is the holomorphic
tangent bundle of ${\mathbf P}$. A \textit{holomorphic connection} on ${\mathbf P}$ is a holomorphic line
subbundle ${\mathcal H}\, \subset\, T{\mathbf P}$ such that the natural homomorphism
\begin{equation}\label{dc}
T_\gamma\oplus {\mathcal H}\, \longrightarrow\, T{\mathbf P}
\end{equation}
is an isomorphism; see \cite{At}.

Let ${\mathcal H}\, \subset\, T{\mathbf P}$ be a holomorphic connection on ${\mathbf P}$.
Let $s\, :\, {\mathbb X}\, \longrightarrow\,{\mathbf P}$ be a holomorphic section of
$\gamma$, meaning $\gamma\circ s\,=\, \text{Id}_{\mathbb X}$. Consider the differential $ds$ of $s$
$$
T {\mathbb X}\, \longrightarrow\, s^*T{\mathbf P}\,=\, (s^*T_\gamma)\oplus (s^*{\mathcal H})\, ,
$$
where the decomposition is the pullback of the decomposition in \eqref{dc}. Let
\begin{equation}\label{dc2}
\widehat{ds}\, :\, T {\mathbb X}\, \longrightarrow\, s^*T_\gamma
\end{equation}
be the homomorphism obtained by composing it with the natural projection
$(s^*T_\gamma)\oplus (s^*{\mathcal H})\, \longrightarrow\, s^*T_\gamma$.

Giving a projective structure on $\mathbb X$ is equivalent to giving a triple $(\gamma,\, {\mathcal H},\, s)$,
where
\begin{itemize}
\item $\gamma\, :\, {\mathbf P}\, \longrightarrow\, \mathbb X$ is a holomorphic ${\mathbb C}{\mathbb P}^1$--bundle,

\item ${\mathcal H}\, \subset\, T{\mathbf P}$ is a holomorphic connection, and

\item $s\, :\, {\mathbb X}\, \longrightarrow\,{\mathbf P}$ is a holomorphic section of
$\gamma$,
\end{itemize}
such that the homomorphism $\widehat{ds}$ in \eqref{dc2} is an isomorphism. More details on this
can be found in \cite{Gu}.

\subsection{Jet bundles}\label{se2.2}

We briefly recall the definition of a jet bundle of a holomorphic vector bundle on $\mathbb X$.
Let
$$
p_j \,\colon\, {\mathbb X}\times{\mathbb X} \,\longrightarrow\,{\mathbb X}\, , \ \ j \,=\, 1,\, 2\,,
$$
be the natural projection to the $j$--th factor. Let
$$
\Delta \,:=\, \{(x,\, x) \,\in\, {\mathbb X}\times{\mathbb X} \,\mid\, x \,\in\,
{\mathbb X}\}\, \subset\, {\mathbb X}\times{\mathbb X}
$$
be the reduced diagonal divisor. For a holomorphic vector bundle $W$ on ${\mathbb X}$, and
any integer $k\, \geq\, 0$, define the $k$--th order jet bundle 
$$
J^k(W) \,:=\, p_{1*} \left((p^*_2W)/(p^*_2W\otimes
{\mathcal O}_{{\mathbb X}\times{\mathbb X}}(-(k+1)\Delta))\right)
\, \longrightarrow \, {\mathbb X}\, .
$$
The natural inclusion of ${\mathcal O}_{{\mathbb X}\times{\mathbb X}}(-(k+1)\Delta)$
in ${\mathcal O}_{{\mathbb X}\times{\mathbb X}}(-k\Delta)$ produces a surjective homomorphism
$J^{k}(W) \,\longrightarrow\, J^{k-1}(W)$. This way we obtain a short exact sequence of
holomorphic vector bundles on ${\mathbb X}$
\begin{equation}\label{e1}
0\, \longrightarrow\, K^{\otimes k}_{\mathbb X}\otimes W \,\longrightarrow\, 
J^{k}(W) \,\longrightarrow\, J^{k-1}(W) \,\longrightarrow\, 0\, ,
\end{equation}
where $K_{\mathbb X}$ is the holomorphic cotangent bundle of ${\mathbb X}$.

For holomorphic vector bundles $W$ and $W'$ on ${\mathbb X}$,
any ${\mathcal O}_{\mathbb X}$--linear homomorphism $W\, \longrightarrow\, W'$ induces a homomorphism
\begin{equation}\label{e2b}
J^i(W)\, \longrightarrow\, J^i(W')
\end{equation}
for every $i\, \geq\, 0$.

For holomorphic vector bundles $W$ and $W'$ on ${\mathbb X}$, any any integer $m\, \geq\, 0$, define the
sheaf of differential operators of order $m$ from $W$ to $W'$
\begin{equation}\label{di1}
\text{Diff}^m_{\mathbb X}(W,\, W') \,:=\, \text{Hom}(J^m(W),\, W') \,\longrightarrow\,{\mathbb X}\, .
\end{equation}
Using the exact sequence in \eqref{e1} we have the short exact sequence
\begin{equation}\label{e2}
0\, \longrightarrow\, \text{Diff}^{m-1}_{\mathbb X}(W,\, W') \, \longrightarrow\,
\text{Diff}^m_{\mathbb X}(W,\, W') \, \stackrel{\sigma}{\longrightarrow}\, (T{\mathbb X})^{\otimes m}
\otimes \text{Hom}(W,\, W') \, \longrightarrow\, 0\, ,
\end{equation}
where $T{\mathbb X}$ is the holomorphic tangent bundle of ${\mathbb X}$. The homomorphism $\sigma$
in \eqref{e2} is called the symbol map.

\begin{remark}\label{rem-j}
Consider the short exact sequences
$$
0\, \longrightarrow\, K_{\mathbb X}\otimes T{\mathbb X}\,=\, {\mathcal O}_{\mathbb X} \,
\longrightarrow\, J^{1}(T{\mathbb X}) \,\longrightarrow\, T{\mathbb X} \,\longrightarrow\, 0
$$
and
$$
0\, \longrightarrow\, K^{\otimes 2}_{\mathbb X}\otimes T{\mathbb X}\,=\, K_{\mathbb X}
\,\longrightarrow\, J^{2}(T{\mathbb X}) \,\longrightarrow\,J^{1}(T{\mathbb X}) \,\longrightarrow\, 0
$$
as in \eqref{e1}. These two together imply that $\bigwedge^3 J^{2}(T{\mathbb X})\,=\, 
K_{\mathbb X}\otimes\bigwedge^2 J^{1}(T{\mathbb X})\,=\, {\mathcal 
O}_{\mathbb X}$. It is straight-forward to check that any for biholomorphism $\beta \, :\, {\mathbb X}\, 
\longrightarrow\, {\mathbb Y}$, the homomorphism $J^2(T{\mathbb X})\, \longrightarrow\,
\beta^* J^2(T{\mathbb Y})$ corresponding to $\beta$
takes the section of $\bigwedge^3 J^{2}(T{\mathbb X})\,=\, {\mathcal
O}_{\mathbb X}$ given by 
the constant function $1$ on ${\mathbb X}$ to the section of $\bigwedge^3 J^{2}(T{\mathbb Y})
\,=\, {\mathcal O}_{\mathbb Y}$ given by the constant function $1$ on ${\mathbb Y}$.
\end{remark}

\subsection{A third order differential operator}\label{se2.3}

We continue with the set-up of Section \ref{se2.1}. Let
$$
{\mathcal T}\, :=\, {\mathbb P}(\mathbb V)\times H^0({\mathbb P}(\mathbb V),\, T{\mathbb P}(\mathbb V))
\, \longrightarrow\,{\mathbb P}(\mathbb V)
$$
be the trivial holomorphic vector bundle of rank three over ${\mathbb P}(\mathbb V)$ with fiber
$H^0({\mathbb P}(\mathbb V),\, T{\mathbb P}(\mathbb V))$. For any integer $j\, \geq\, 1$, let
\begin{equation}\label{e0}
\psi_j\, :\, {\mathcal T}\, \longrightarrow\, J^j(T{\mathbb P}(\mathbb V))
\end{equation}
be the holomorphic ${\mathcal O}_{{\mathbb P}(\mathbb V)}$--linear
map that sends any $(x,\, s)\, \in\, {\mathbb P}(\mathbb V)\times H^0({\mathbb P}(\mathbb V),
\, T{\mathbb P}(\mathbb V))$ to the restriction of the section $s$ to the $j$--th order infinitesimal
neighborhood of the point $x\, \in\, {\mathbb P}(\mathbb V)$.

\begin{lemma}\label{lem1}
The homomorphism $\psi_2$ in \eqref{e0} is an isomorphism.
\end{lemma}

\begin{proof}
If $(x,\, s)\, \in\, \text{kernel}(\psi_2(x))$, then
$$
s \, \in\, H^0({\mathbb P}(\mathbb V),\, {\mathcal O}_{{\mathbb P}(\mathbb V)}(-3x)\otimes
T{\mathbb P}(\mathbb V))\, .
$$
But $H^0({\mathbb P}(\mathbb V),\, {\mathcal O}_{{\mathbb P}(\mathbb V)}(-3x)\otimes
T{\mathbb P}(\mathbb V))\,=\, 0$, because $\text{degree}({\mathcal O}_{{\mathbb P}(\mathbb V)}(-3x)\otimes
T{\mathbb P}(\mathbb V))\, <\, 0$. So the homomorphism $\psi_2$ is fiberwise injective. This implies that
$\psi_2$ is an isomorphism, because we have $\text{rank}({\mathcal T})\,=\, \text{rank}(J^2(T{\mathbb P}
(\mathbb V))$.
\end{proof}

\begin{lemma}\label{lem2}
There is a canonical holomorphic differential operator $\delta_0$ of order three from
$T{\mathbb P}(\mathbb V))$ to $K^{\otimes 2}_{{\mathbb P}(\mathbb V)}$. The symbol of $\delta_0$
is the section of
$$
(T{\mathbb P}(\mathbb V))^{\otimes 3}\otimes{\rm Hom}\Big(T{\mathbb P}(\mathbb V),\,
K^{\otimes 2}_{{\mathbb P}(\mathbb V)}\Big)\,=\, {\mathcal O}_{{\mathbb P}(\mathbb V)}
$$
given by the constant function $1$ on ${\mathbb P}(\mathbb V)$.
\end{lemma}

\begin{proof}
Consider the short exact sequence
\begin{equation}\label{e3}
0\, \longrightarrow\, K^{\otimes 3}_{{\mathbb P}(\mathbb V)}\otimes T{\mathbb P}(\mathbb V)
\,=\, K^{\otimes 2}_{{\mathbb P}(\mathbb V)} \,\stackrel{\iota_0}{\longrightarrow}\,
J^3(T{\mathbb P}(\mathbb V)) \,\longrightarrow\, J^2(T{\mathbb P}(\mathbb V)) \,\longrightarrow\, 0
\end{equation}
in \eqref{e1}. Using Lemma \ref{lem1}, define the homomorphism
$$
\psi_3\circ (\psi_2)^{-1}\, :\, J^2(T{\mathbb P}(\mathbb V))\,\longrightarrow\,J^3(T{\mathbb P}(\mathbb V))\, ,
$$
where the homomorphisms
$\psi_j$ are constructed in \eqref{e0}. This homomorphism $\psi_3\circ (\psi_2)^{-1}$
is a holomorphic splitting of the exact sequence in \eqref{e3}. 
In other words, there is a unique surjective homomorphism
$$
\delta_0\, :\, J^3(T{\mathbb P}(\mathbb V))\,\longrightarrow\, K^{\otimes 2}_{{\mathbb P}(\mathbb V)}
$$
such that $\text{kernel}(\delta_0)\,=\, \text{image}(\psi_3\circ (\psi_2)^{-1})$ and
\begin{equation}\label{s}
\delta_0\circ\iota_0\,=\, \text{Id}_{K^{\otimes 2}_{{\mathbb P}(\mathbb V)}}\, ,
\end{equation}
where $\iota_0$ is the homomorphism in \eqref{e3}.

{}From the definition in \eqref{di1} it follows
that $$\delta_0\,\in\, H^0({\mathbb P}(\mathbb V),\, \text{Diff}^3_{{\mathbb P}(\mathbb V)}
(T{\mathbb P}(\mathbb V),\, K^{\otimes 2}_{{\mathbb P}(\mathbb V)}))\, .$$ Also, from \eqref{s} it
follows immediately that $\sigma(\delta_0)\,=\, 1$, where $\sigma$ is the symbol homomorphism
in \eqref{e2}.
\end{proof}

The trivialization of $J^2(T{\mathbb P}(\mathbb V))$ given by Lemma \ref{lem1} produces a
holomorphic connection on $J^2(T{\mathbb P}(\mathbb V))$; let
\begin{equation}\label{d}
\mathbb{D}_0\, :\, J^2(T{\mathbb P}(\mathbb V))\, \longrightarrow\, J^2(T{\mathbb P}(\mathbb V))\otimes
K_{{\mathbb P}(\mathbb V)}
\end{equation}
be this holomorphic connection on $J^2(T{\mathbb P}(\mathbb V))$. Note that any holomorphic
connection on a Riemann surface is automatically flat (see \cite{At} for holomorphic connections).

\begin{remark}\label{rem-1}
Let $U\, \subset\, {\mathbb P}(\mathbb V)$ be an open subset and
$$
s\, \in\, H^0(U,\, J^2(T{\mathbb P}(\mathbb V))\vert_U)\,=\, H^0(U,\, J^2(TU))
$$ a flat section for the connection
$\mathbb{D}_0$ in \eqref{d}. Since $\psi_2$ is an isomorphism (see Lemma \ref{lem1}), it follows that
the section $s'\, \in\, H^0(U,\, TU)$ given by $s$ using the natural projection
$J^2(T{\mathbb P}(\mathbb V))\, \longrightarrow\, T{\mathbb P}(\mathbb V)$ (see \eqref{e1}) has
the property that the section of $J^2(T{\mathbb P}(\mathbb V))\vert_U$ corresponding to $s'$ coincides
with $s$. If $U$ is connected, then $s'$ extends to a holomorphic
section of $T{\mathbb P}(\mathbb V)$ over ${\mathbb P}(\mathbb V)$.
\end{remark}

Let $\mathfrak{sl}(\mathbb V)$ be the Lie algebra of $\text{PGL}(\mathbb V)$; it consists of
endomorphisms of $\mathbb V$ of trace zero. Using the action of $\text{PGL}(\mathbb V)$ on
${\mathbb P}(\mathbb V)$ we get a homomorphism
\begin{equation}\label{a0}
\alpha_0\, :\, \mathfrak{sl}(\mathbb V)\, \longrightarrow\, H^0({\mathbb P}(\mathbb V),\, T{\mathbb P}(\mathbb V))\, .
\end{equation}
This $\alpha_0$ is an isomorphism, because it is injective and $\dim \mathfrak{sl}(\mathbb V)\,
=\, \dim H^0({\mathbb P}(\mathbb V),\, T{\mathbb P}(\mathbb V))$. Note that
$H^0({\mathbb P}(\mathbb V),\, T{\mathbb P}(\mathbb V))$ has the structure of a Lie algebra given by the
Lie bracket operation of vector fields.
The homomorphism $\alpha_0$ in \eqref{a0} is in fact an isomorphism of Lie algebras.
Therefore, from Lemma \ref{lem1} it follows that the fibers of 
$J^2(T{\mathbb P}(\mathbb V))$ are identified with the Lie algebra $\mathfrak{sl}(\mathbb V)$.
In particular, the fibers of $J^2(T{\mathbb P}(\mathbb V))$ are Lie algebras.
For any $x\, \in\, {\mathbb P}(\mathbb V))$, and $v,\, w\, \in\, J^2(T{\mathbb P}(\mathbb V))_x$, let
\begin{equation}\label{d1}
[v,\, w]'\, \in\, J^2(T{\mathbb P}(\mathbb V))_x
\end{equation}
be the Lie bracket operation on the fiber $J^2(T{\mathbb P}(\mathbb V))_x$.

\begin{remark}\label{rem0}
Let $s$ and $t$ be two holomorphic vector fields on an open subset $U\, \subset\, {\mathbb 
P}(\mathbb V)$. The holomorphic section of $J^2(T{\mathbb P}(\mathbb V))\vert_U$ defined by $s$ 
and $t$ will be denoted by $\widehat{s}$ and $\widehat{t}$ respectively. It should be clarified 
that the holomorphic section $\widehat{[s,t]}$ of $J^2(T{\mathbb P}(\mathbb V))\vert_U$ given by 
the Lie bracket $[s,\, t]$ of vector fields does not in general coincide with the section 
$[\widehat{s},\, \widehat{t}]'$ defined by \eqref{d1}. The reason for it is
that the operation in \eqref{d1} is constructed 
using the finite dimensional space consisting
of global holomorphic sections of $T{\mathbb P}(\mathbb V)$, 
while the operation of Lie bracket of vector fields is constructed locally. However if $s$ and $t$ are
such that $\widehat{s}$ and $\widehat{t}$ are flat sections for the holomorphic connection $\mathbb{D}_0$ 
on $J^2(T{\mathbb P}(\mathbb V))$ in \eqref{d}, then the holomorphic section $\widehat{[s,\,t]}$ of 
$J^2(T{\mathbb P}(\mathbb V))\vert_U$ given by the Lie bracket of vector fields $[s,\, t]$ does 
coincide with the section $[\widehat{s},\, \widehat{t}]'$ defined by \eqref{d1}. Indeed, this 
follows from the fact that these $s$ and $t$ are restrictions of global vector fields
on ${\mathbb P}(\mathbb V)$, if $U$ is connected; see Remark \ref{rem-1}.
\end{remark}

\begin{remark}\label{rem1}
The holomorphic connection $\mathbb{D}_0$ in \eqref{d} on $J^2(T{\mathbb P}(\mathbb V))$ preserves the Lie
algebra structure on the fibers of $J^2(T{\mathbb P}(\mathbb V))$ given in \eqref{d1}. This means that
$$
i_u \mathbb{D}_0([s,\, t]')\,=\, [i_u\mathbb{D}_0(s),\, t]'+
[s,\, i_u\mathbb{D}_0(t)]'
$$
for locally defined holomorphic sections $s$, $t$ and $u$ of $J^2(T{\mathbb P}(\mathbb V))$,
where $i_u$ is the contraction of $1$--forms by $u$. In particular,
the local system on ${\mathbb P}(\mathbb V)$ given by the flat sections for the connection
$\mathbb{D}_0$ is closed under the Lie bracket operation in \eqref{d1}.
Note that using Remark \ref{rem-1} we may construct a Lie algebra structure on the fibers
of $J^2(T{\mathbb P}(\mathbb V))$. Indeed, for $v,\, w\, \in\, J^2(T{\mathbb P}(\mathbb V))_x$, let
$\widetilde{v},\, \widetilde{w}$ be the flat sections of $J^2(T{\mathbb P}(\mathbb V))$, for the
connection $\mathbb{D}_0$, defined around $x$ such that $\widetilde{v}(x)\,=\, v$ and
$\widetilde{w}(x)\,=\, w$. Let $\widetilde{v}'$ (respectively, $\widetilde{w}'$) be the holomorphic
sections of $T{\mathbb P}(\mathbb V)$ defined around $x$
given by $\widetilde{v}$ (respectively, $\widetilde{w}$) using the natural projection
$J^2(T{\mathbb P}(\mathbb V))\, \longrightarrow\, T{\mathbb P}(\mathbb V)$ (see \eqref{e1}
and Remark \ref{rem-1}). Now define $[v,\, w]$ to be the element of
$J^2(T{\mathbb P}(\mathbb V))_x$ given by the locally defined
section $[\widetilde{v}',\,\widetilde{w}']$ of $T{\mathbb P}(\mathbb V)$. From Remark \ref{rem0}
it follows that this Lie algebra structure on the fibers of $J^2(T{\mathbb P}(\mathbb V))$ coincides with
the one in \eqref{d1}.
\end{remark}

\begin{remark}\label{rem2}
Let ${\mathbb L}_0$ be the complex local system on ${\mathbb P}(\mathbb V)$ given by the sheaf of 
solutions of the differential operator $\delta_0$ in Lemma \ref{lem2}. From the construction of 
$\delta_0$ it is straight-forward to deduce that ${\mathbb L}_0$ is identified with the local 
system given by the sheaf of flat sections of $J^2(T{\mathbb P}(\mathbb V))$ for the connection 
$\mathbb{D}_0$ in \eqref{d}. Therefore, from Remark \ref{rem1} we conclude that the stalks of the 
complex local system ${\mathbb L}_0$ are closed under the Lie bracket operation of vector fields. 
Moreover, the stalks of the local system ${\mathbb L}_0$ are identified with the Lie algebra
$\mathfrak{sl}(\mathbb V)$.
\end{remark}

\begin{proposition}\label{prop1}\mbox{}
\begin{enumerate}
\item Let $\mathbb X$ be a connected Riemann surface equipped with a projective structure $\mathcal 
P$. Then $\mathcal P$ produces a holomorphic connection, which will be called ${\mathbb D} 
({\mathcal P})$, on $J^2(T{\mathbb X})$. For any open subset $U\, \subset\, \mathbb X$, and any
section $s\, \in\, H^0(U,\, J^2(T{\mathbb X})\vert_U)\,=\, H^0(U,\, J^2(TU))$ flat for the
connection ${\mathbb D}({\mathcal P})$, there is a unique holomorphic section of $TU$ that produces
$s$. The space of section of $TU$ given by the flat sections of $J^2(U)$ is closed under the usual Lie bracket
operation of vector fields. The stalks for the local system on $\mathbb X$ given by 
the sheaf of flat sections for ${\mathbb D}({\mathcal P})$ are closed under the usual Lie bracket 
operation of vector fields, and moreover the stalks are isomorphic to the Lie algebra $\mathfrak{sl}(\mathbb V)$.

\item The projective structure $\mathcal P$ also produces a canonical holomorphic differential operator
$\delta({\mathcal P})\, \in\, H^0({\mathbb X},\, {\rm Diff}^3_{\mathbb X}(T{\mathbb X},\,
K^{\otimes 2}_{\mathbb X}))$ whose symbol is the constant function $1$ on ${\mathbb X}$.

\item The local system on $\mathbb X$ gives by the sheaf of flat sections for ${\mathbb 
D}({\mathcal P})$ is identified with the local system given by the sheaf of solutions of 
$\delta({\mathcal P})$. This sheaf of solutions of $\delta({\mathcal P})$ is closed under the Lie 
bracket operation of vector fields.

\item The connection on $\bigwedge^3 J^2(T{\mathbb X})\,=\, {\mathcal O}_{\mathbb X}$ (see Remark \ref{rem-j})
induced by ${\mathbb D}({\mathcal P})$ coincides with the trivial connection on ${\mathcal O}_{\mathbb X}$
given by the de Rham differential $d$.
\end{enumerate}
\end{proposition}

\begin{proof}
The action of $\text{PGL}(\mathbb V)$ on ${\mathbb P}(\mathbb V)$ produces actions of 
$\text{PGL}(\mathbb V)$ on $J^j(T{\mathbb P}(\mathbb V))$, $j\, \geq\, 0$, and $H^0({\mathbb P}(\mathbb V),\, 
T{\mathbb P}(\mathbb V))$. The homomorphism $\psi_j$ in \eqref{e0} is clearly equivariant for the 
actions of $\text{PGL}(\mathbb V)$, with $\text{PGL}(\mathbb V)$ acting diagonally on ${\mathbb 
P}(\mathbb V)\times H^0({\mathbb P}(\mathbb V),\, T{\mathbb P}(\mathbb V))$. Therefore, from 
Lemma \ref{lem1} we get a holomorphic connection on $J^2(T{\mathbb X})$. To explain this with
more details, take a 
holomorphic coordinate atlas $\{(U_i,\, \phi_i)\}_{i\in I}$ in the equivalence class defining 
$\mathcal P$. Using $\phi_i$, the holomorphic connection ${\mathbb D}_0\vert_{\phi_i(U_i)}$ in 
\eqref{d} on $J^2(T{\mathbb P}(\mathbb V))\vert_{\phi_i(U_i)}$ produces a holomorphic connection 
on $J^2(T{\mathbb X})\vert_{U_i}$. Using the above $\text{PGL}(\mathbb V)$--equivariance 
property, these locally defined connections on $J^2(T{\mathbb X})\vert_{U_i}$, $~i\, \in\, I$, 
patch together compatibly on the intersections of the open
subsets to produce a holomorphic connection ${\mathbb D}({\mathcal P})$ on 
$J^2(T{\mathbb X})$.

Take any flat section $s\, \in\, H^0(U,\, J^2(T{\mathbb X})\vert_U)\,=\, H^0(U,\, J^2(TU))$ as in the
first statement of the proposition. From Remark \ref{rem-1} we conclude that the section
of $TU$ given by $s$, using the natural projection $J^2(TU)\, \longrightarrow\, TU$,
actually produces $s$. From Remark \ref{rem0} it follows that
the space of section of $TU$ given by the flat sections of $J^2(TU)$ is closed under the usual Lie bracket
operation of vector fields. Consequently,
the stalks for the local system on $\mathbb X$ given by the sheaf of flat sections for ${\mathbb 
D}({\mathcal P})$ are closed under the Lie bracket operation, and they are isomorphic to the Lie 
algebra $\mathfrak{sl}(\mathbb V)$, because the connection ${\mathbb D}_0$ has these properties; see 
Remark \ref{rem1}.

Similarly, the second statement of the proposition follows from the fact that
the differential operator $\delta_0$ in Lemma \ref{lem2} is $\text{PGL}(\mathbb V)$--equivariant.
Indeed, given a holomorphic coordinate atlas $\{(U_i,\, \phi_i)\}_{i\in I}$ as above, we have
a differential operator on each $U_i$ given by $\delta_0$ using the coordinate function $\phi_i$.
These differential operator patch together compatibly to produce a differential operator
$$\delta({\mathcal P})\, \in\, H^0({\mathbb X},\, {\rm Diff}^3_{\mathbb X}(T{\mathbb X},\,
K^{\otimes 2}_{\mathbb X}))\, .$$

The third statement follows from Remark \ref{rem2} and the first statement of the proposition. For
any $s\, \in\, H^0(U,\, TU)$, where $U\, \subset\, {\mathbb X}$ is an open subset,
with $\delta({\mathcal P})(s)\,=\, 0$, the section of $J^2(T{\mathbb X})\vert_U$ given by $s$ is
flat with respect to the connection ${\mathbb D}({\mathcal P})$. Conversely, given a flat section
$s_1$ of $J^2(T{\mathbb X})\vert_U$, the section $s_2 \, \in\, H^0(U,\, TU)$, given by $s_1$ using the
natural projection $J^2(T{\mathbb X})\, \longrightarrow\, T{\mathbb X}$ (see \eqref{e1}), satisfies
the equation $\delta({\mathcal P})(s_2)\,=\, 0$.

{}From Remark \ref{rem-j} we know that if $\phi\, :\, {\mathbb U}_1\, \longrightarrow\, {\mathbb U}_2$
is a biholomorphism between two open subsets of ${\mathbb P}(\mathbb V)$, then the isomorphism
$\bigwedge^3 J^{2}(T{\mathbb U}_1)\, \longrightarrow\, \bigwedge^3 J^{2}(T{\mathbb U}_2)$ induced by
$\phi$ takes the section of $\bigwedge^3 J^{2}(T{\mathbb U}_1)$ given by the constant function $1$
on $U_1$ to the section of $\bigwedge^3 J^{2}(T{\mathbb U}_2)$ given by the constant function $1$
on $U_2$. In particular, this
holds for $\phi\, \in\, \text{PGL}(\mathbb V)\,=\, \text{Aut}({\mathbb P}(\mathbb V))$.
The connection on $\bigwedge^3 J^{2}(T{\mathbb P}(\mathbb V))$ induced by 
the connection $\mathbb{D}_0$ on $J^2(T{\mathbb P}(\mathbb V))$ coincides with the trivial connection on
$\bigwedge^3 J^{2}(T{\mathbb P}(\mathbb V))$, because ${\mathbb P}(\mathbb V)$ is simply connected.
The fourth statement of the proposition follows from these.
\end{proof}

\subsection{Killing form and holomorphic connection}\label{se2.4}

Recall that the homomorphism
$\alpha_0$ in \eqref{a0} is a Lie algebra isomorphism between $H^0({\mathbb P}(\mathbb V),\,
T{\mathbb P}(\mathbb V))$ and $\mathfrak{sl}(\mathbb V)$. Consider the Killing form $\widehat B$ on the
Lie algebra $H^0({\mathbb P}(\mathbb V),\,
T{\mathbb P}(\mathbb V))$. Using the isomorphism $\psi_2$ in Lemma \ref{lem1}, this symmetric bilinear form
$\widehat B$ produces a fiberwise nondegenerate symmetric bilinear form
\begin{equation}\label{eB}
B_0\, \in\, H^0({\mathbb P}(\mathbb V),\, \text{Sym}^2(J^2(T{\mathbb P}(\mathbb V)))^*)
\,=\, \text{Sym}^2(H^0({\mathbb P}(\mathbb V),\, T{\mathbb P}(\mathbb V)))^*\, .
\end{equation}
Recall that each fiber of $J^2(T{\mathbb P}(\mathbb V))$ is the Lie algebra $\mathfrak{sl}(\mathbb V)$;
the form $B_0$ in \eqref{eB} is the fiberwise Killing form.
The symmetric form $B_0$ is preserved by the holomorphic connection $\mathbb{D}_0$ on
$J^2(T{\mathbb P}(\mathbb V))$ constructed in \eqref{d}. Indeed, this follows immediately from the fact that
both $\mathbb{D}_0$ and $B_0$ are constants with respect to the trivialization of $J^2(T{\mathbb P}(\mathbb V))$
given by $\psi_2$ in Lemma \ref{lem1}.

The vector bundle $J^2(T{\mathbb P}(\mathbb V))$ has a filtration of holomorphic subbundles
\begin{equation}\label{f1}
F^{\mathbb P}_1\,:=\, K_{{\mathbb P}(\mathbb V)}\, \subset\, F^{\mathbb P}_2\, \subset\,
J^2(T{\mathbb P}(\mathbb V))\, ,
\end{equation}
where $F^{\mathbb P}_2$ is the kernel of the composition
$$
J^2(T{\mathbb P}(\mathbb V))\, \longrightarrow\,J^1(T{\mathbb P}(\mathbb V)) \, \longrightarrow\,
T{\mathbb P}(\mathbb V)
$$
of the two projections in the two short exact sequences in Remark \ref{rem-j}; the subbundle
$K_{{\mathbb P}(\mathbb V)}\, \subset\, J^2(T{\mathbb P}(\mathbb V))$ in \eqref{f1} is the one in the
second of the two short exact sequences in Remark \ref{rem-j}. In particular, we have
$\text{rank}(F^{\mathbb P}_j)\,=\,j$. For any point $x\, \in\, {\mathbb P}(\mathbb V)$,
the fiber $(F^{\mathbb P}_1)_x$ is a nilpotent subalgebra of the Lie algebra $J^2(T{\mathbb P}(\mathbb V))_x
\,=\,\mathfrak{sl}(\mathbb V)$. Moreover, the fiber $(F^{\mathbb P}_2)_x$ is the unique Borel subalgebra of
$J^2(T{\mathbb P}(\mathbb V))_x$ containing $(F^{\mathbb P}_1)_x$. Consequently, we have
\begin{equation}\label{f2}
B_0(F^{\mathbb P}_1\otimes F^{\mathbb P}_1)\,=\, 0 \ \ \text{ and }\ \ (F^{\mathbb P}_1)^\perp\,=\,
F^{\mathbb P}_2\, ,
\end{equation}
where $(F^{\mathbb P}_1)^\perp$ denotes the orthogonal bundle for $F^{\mathbb P}_1$ with respect to the
form $B_0$ in \eqref{eB}.

Given a holomorphic vector bundle $W$ on a Riemann surface $\mathbb X$, a holomorphic connection $\mathcal D$ on 
$W$, and a holomorphic subbundle $W'\,\subset\, W$, the composition of homomorphisms
$$
W\, \stackrel{\mathcal D}{\longrightarrow}\, W\otimes K_{\mathbb X} \,
\stackrel{q_{W'}\otimes{\rm Id}}{\longrightarrow}\,
(W/W')\otimes K_{\mathbb X}\, ,
$$
where $q_{W'}\, :\, W\,\longrightarrow\,W/W'$ is the natural quotient map, defines a holomorphic 
section of $\text{Hom}(W',\, (W/W'))\otimes K_{\mathbb X}$. This element of $H^0({\mathbb X},\, 
\text{Hom}(W',\, (W/W'))\otimes K_{\mathbb X})$ is called the \textit{second fundamental form} of 
$W'$ for the connection $\mathcal D$. If ${\mathcal D}(W')\, \subset\, W''\otimes K_{\mathbb X}$, 
where $W''$ is a holomorphic subbundle of $W$ containing $W'$, then the second 
fundamental form of $W'$ for $\mathcal D$ is clearly given by a holomorphic section
\begin{equation}\label{j1}
\zeta_1\,\in\, H^0({\mathbb X},\, \text{Hom}(W',\, W''/W')\otimes K_{\mathbb X})
\end{equation}
using the natural inclusion of $\text{Hom}(W',\, W''/W')\otimes 
K_{\mathbb X}$ in $\text{Hom}(W',\, W/W')\otimes K_{\mathbb X}$. Also, in this case, the second 
fundamental form of $W''$ for $\mathcal D$ is given by a holomorphic section
\begin{equation}\label{j2}
\zeta_2\,\in\, H^0({\mathbb X},\,\text{Hom}(W''/W',\, W/W'')\otimes K_{\mathbb X})
\end{equation}
through the natural inclusion map 
$$H^0({\mathbb X},\, \text{Hom}(W''/W',\, W/W'')\otimes K_{\mathbb X})\, \hookrightarrow\,
H^0({\mathbb X},\, \text{Hom}(W'',\, W/W'')\otimes K_{\mathbb X})\, .$$

For the filtration in \eqref{f1} of $J^2(T{\mathbb P}(\mathbb V))$ equipped with the holomorphic connection
$\mathbb{D}_0$ in \eqref{d}, we have
$$
\mathbb{D}_0(F^{\mathbb P}_1)\,=\, F^{\mathbb P}_2\otimes K_{{\mathbb P}(\mathbb V)}
 \ \ \text{ and }\ \ \mathbb{D}_0(F^{\mathbb P}_2)\,=\,J^2(T{\mathbb P}(\mathbb V))\otimes
K_{{\mathbb P}(\mathbb V)}\, .
$$
These follow from a straight-forward computation.

Let
$$
S(F^{\mathbb P}_1, \mathbb{D}_0)\, \in \, H^0({\mathbb P}(\mathbb V),\, \text{Hom}(F^{\mathbb P}_1,\,
F^{\mathbb P}_2/F^{\mathbb P}_1)\otimes K_{{\mathbb P}(\mathbb V)})\,=\, H^0({\mathbb P}(\mathbb V),\,
{\mathcal O}_{{\mathbb P}(\mathbb V)})
$$
be the second fundamental form of $F^{\mathbb P}_1$ for the connection $\mathbb{D}_0$ (see \eqref{j1}).
Similarly, let
$$
S(F^{\mathbb P}_2, \mathbb{D}_0)\, \in \, H^0({\mathbb P}(\mathbb V),\, \text{Hom}(F^{\mathbb P}_2/
F^{\mathbb P}_1,\, J^2(T{\mathbb P}(\mathbb V))/F^{\mathbb P}_2)\otimes K_{{\mathbb P}(\mathbb V)})\,=\,
H^0({\mathbb P}(\mathbb V),\, {\mathcal O}_{{\mathbb P}(\mathbb V)})
$$
be the section that gives the second fundamental form of $F^{\mathbb P}_2$ for the connection
$\mathbb{D}_0$ (see \eqref{j2}). It is straight-forward to check that both $S(F^{\mathbb P}_1, \mathbb{D}_0)$
and $S(F^{\mathbb P}_2, \mathbb{D}_0)$ coincide with the element of $H^0({\mathbb P}(\mathbb V),\,
{\mathcal O}_{{\mathbb P}(\mathbb V)})$ given by the constant function $1$ on ${\mathbb P}(\mathbb V)$.

\section{Differential operators, connections and projective structures}\label{se3}

\subsection{Differential operators and connections}\label{se3.1}

For a holomorphic vector bundle $W$ on a Riemann surface $\mathbb X$, there is a tautological fiberwise injective
holomorphic homomorphism
\begin{equation}\label{c}
J^{i+j}(W)\, \longrightarrow\, J^i(J^j(W))
\end{equation}
for every $i,\, j\, \geq\, 0$. On $\mathbb X$, we have the
commutative diagram of holomorphic homomorphisms
\begin{equation}\label{e4}
\begin{matrix}
&& 0 && 0\\
&& \Big\downarrow && \Big\downarrow\\
0 & \longrightarrow & K^{\otimes 2}_{\mathbb X} & \stackrel{\iota}{\longrightarrow} &
J^{3}(T{\mathbb X}) & \stackrel{q}{\longrightarrow} & J^{2}(T{\mathbb X}) & \longrightarrow && 0\\
&& ~\, ~\Big\downarrow l && ~\,~ \Big\downarrow \lambda && \Vert\\
0 & \longrightarrow & J^{2}(T{\mathbb X}) \otimes K_{\mathbb X} & \stackrel{\iota'}{\longrightarrow} &
J^1(J^{2}(T{\mathbb X})) & \stackrel{q'}{\longrightarrow} & J^{2}(T{\mathbb X}) & \longrightarrow && 0\\
&& \Big\downarrow &&~\,~\Big\downarrow\mu \\
&& J^{1}(T{\mathbb X})\otimes K_{\mathbb X} & \stackrel{=}{\longrightarrow} &
J^{1}(T{\mathbb X})\otimes K_{\mathbb X}\\
&& \Big\downarrow && \Big\downarrow\\
&& 0 && 0
\end{matrix}
\end{equation}
where the horizontal short exact sequences are as in \eqref{e1}, the vertical short exact sequence
in the left is the short exact sequence in \eqref{e1} tensored with $K_{\mathbb X}$, and $\lambda$ is the
homomorphism in \eqref{c}; the homomorphism $\mu$ in \eqref{e4} is described below.

The projection $J^{2}(T{\mathbb X})\, \longrightarrow\,
J^{1}(T{\mathbb X})$ in \eqref{e1} induces a homomorphism
\begin{equation}\label{c2}
f_1\, :\, J^1(J^{2}(T{\mathbb X}))\, \longrightarrow\,J^1(J^{1}(T{\mathbb X}))
\end{equation}
(see \eqref{e2b}); set $W$ and $W'$ in \eqref{e2b} to be
$J^{2}(T{\mathbb X})$ and $J^{1}(T{\mathbb X})$ respectively to get $f_1$. On the other hand, let
$$
f'_2\, :\, J^1(J^{2}(T{\mathbb X}))\, \longrightarrow\, J^{2}(T{\mathbb X})
$$
be the projection in \eqref{e1}. Composing $f'_2$ with the homomorphism
$J^{2}(T{\mathbb X})\, \longrightarrow\, J^1(J^{1}(T{\mathbb X}))$ in
\eqref{c} we obtain a homomorphism
$$
f_2\, :\, J^1(J^{2}(T{\mathbb X}))\, \longrightarrow\, J^1(J^{1}(T{\mathbb X}))\, .
$$
The composition of homomorphisms
$$
J^1(J^{2}(T{\mathbb X}))\, \stackrel{f_2}{\longrightarrow}\, J^1(J^{1}(T{\mathbb X}))
\, \stackrel{f_3}{\longrightarrow}\, J^{1}(T{\mathbb X})\, ,
$$
where $f_3$ is projection in \eqref{e1},
coincides with the composition of homomorphisms
$$
J^1(J^{2}(T{\mathbb X}))\, \stackrel{f_1}{\longrightarrow}\, J^1(J^{1}(T{\mathbb X}))
\, \stackrel{f_3}{\longrightarrow}\, J^{1}(T{\mathbb X})\, ,
$$
where $f_1$ is the homomorphism in \eqref{c2}. Therefore, from \eqref{e1} we have the homomorphism
$$
\mu\ :=\, f_1-f_2\, :\, J^1(J^{2}(T{\mathbb X}))\, \longrightarrow\, J^{1}(T{\mathbb X})
\otimes K_{\mathbb X}\, ,
$$
where $f_1$ and $f_2$ are constructed above. This homomorphism $\mu$ is 
the one in \eqref{e4}.

Let
\begin{equation}\label{eta}
\eta\, \in\, H^0({\mathbb X},\, \text{Diff}^3_{\mathbb X}(T{\mathbb X},\,
K^{\otimes 2}_{\mathbb X}))\,=\, H^0({\mathbb X},\, K^{\otimes 2}_{\mathbb X}\otimes J^{3}(T{\mathbb X})^*)
\end{equation}
be a differential operator whose symbol is the constant function $1$
on ${\mathbb X}$. This means that $\eta$ gives a holomorphic
splitting of the top horizontal exact sequence in \eqref{e4}. Let
\begin{equation}\label{we}
\widehat{\eta}\, :\, J^{2}(T{\mathbb X})\, \longrightarrow\, J^{3}(T{\mathbb X})
\end{equation}
be the corresponding splitting homomorphism, meaning
\begin{itemize}
\item $\widehat{\eta}(J^{2}(T{\mathbb X}))\,=\, \text{kernel}(J^{3}(T{\mathbb X})\stackrel{\eta}{\rightarrow}
K^{\otimes 2}_{\mathbb X})$, and

\item $q\circ\widehat{\eta}\,=\, \text{Id}_{J^{2}(T{\mathbb X})}$, where $q$ is the projection in
\eqref{e4}.
\end{itemize}
{}From the commutativity of \eqref{e4} we conclude that the homomorphism
\begin{equation}\label{c3}
\lambda\circ \widehat{\eta}\, :\, J^{2}(T{\mathbb X})\, \longrightarrow\, J^1(J^{2}(T{\mathbb X}))\, ,
\end{equation}
where $\lambda$ is the homomorphism in \eqref{e4}, satisfies the equation
$$
q'\circ (\lambda\circ \widehat{\eta})\,=\, \text{Id}_{J^{2}(T{\mathbb X})}\, ,$$
where $q'$ is the projection in \eqref{e4}. Consequently, the homomorphism $\lambda\circ \widehat{\eta}$
defines a holomorphic connection on $J^{2}(T{\mathbb X})$ (see \cite{At}).

Let ${\rm Conn}(J^2(T{\mathbb X}))$ denote the space of all holomorphic connections on 
$J^2(T{\mathbb X})$.

We summarize the above construction in the following lemma.

\begin{lemma}\label{lem3}
Consider the subset
$$
H^0({\mathbb X},\, {\rm Diff}^3_{\mathbb X}(T{\mathbb X},\,
K^{\otimes 2}_{\mathbb X}))_0\, \subset\,
H^0({\mathbb X},\, {\rm Diff}^3_{\mathbb X}(T{\mathbb X},\,
K^{\otimes 2}_{\mathbb X}))
$$
defined by the differential operators whose symbol is the constant
function $1$ on $\mathbb X$. There is a natural map
$$
\varpi\, :\, H^0({\mathbb X},\, {\rm Diff}^3_{\mathbb X}(T{\mathbb X},\,
K^{\otimes 2}_{\mathbb X}))_0\, \longrightarrow\,{\rm Conn}(J^2(T{\mathbb X}))
$$
that sends any $\eta$ as in 
\eqref{eta} to the connection $\lambda\circ \widehat{\eta}$ in \eqref{c3}.
\end{lemma}

We shall describe the image of the map $\varpi$ in Lemma \ref{lem3}.

Consider the two short exact sequences in Remark \ref{rem-j}. Let
\begin{equation}\label{ef}
K_{\mathbb X} \, :=\, \text{kernel}(\mu_0)\, \subset\, F_2\, :=\, \mu^{-1}_0({\mathcal O}_{\mathbb X})
\, \subset\, J^{2}(T{\mathbb X})
\end{equation}
be the filtration of holomorphic subbundles given by them,
where $\mu_0\, :\, J^{2}(T{\mathbb X})\, \longrightarrow\, J^{1}(T{\mathbb X})$ is the projection in
\eqref{e1}; see \eqref{f1}.

As before, ${\rm Conn}(J^2(T{\mathbb X}))$ denotes the space of all holomorphic connections on 
$J^2(T{\mathbb X})$.

\begin{definition}\label{def1}
A holomorphic connection ${\mathcal D}\, \in\, {\rm Conn}(J^2(T{\mathbb X}))$ will be
called an \textit{oper connection} if the following three conditions hold:
\begin{itemize}
\item ${\mathcal D}(K_{\mathbb X})\,=\, F_2\otimes K_{\mathbb X}$ (see \eqref{ef}),

\item the second fundamental form of $K_{\mathbb X}$ for $\mathcal D$, which, by the first 
condition, is a holomorphic section of $\text{Hom}(K_{\mathbb X},\, {\mathcal O}_{\mathbb 
X})\otimes K_{\mathbb X}\,=\, {\mathcal O}_{\mathbb X}$ (see \eqref{j1}), coincides with the constant function 
$1$ on ${\mathbb X}$, and

\item the holomorphic section of $\text{Hom}(F_2/K_{\mathbb X},\,
J^{2}(T{\mathbb X})/F_2)\otimes K_{\mathbb X}\,=\, {\mathcal O}_{\mathbb X}$ that gives the
second fundamental form of $F_2$ for $\mathcal D$ --- see \eqref{j2} --- coincides with the constant function
$1$ on ${\mathbb X}$.
\end{itemize}
\end{definition}

See \cite{BeDr} for general opers; the oper connections in Definition \ref{def1} are
$\text{GL}(3, {\mathbb C})$--opers on $\mathbb X$.

\begin{lemma}\label{lemm0}
Take any $\eta\, \in\, H^0({\mathbb X},\, {\rm Diff}^3_{\mathbb X}(T{\mathbb X},\,
K^{\otimes 2}_{\mathbb X}))_0$ (see Lemma \ref{lem3}), and let
$$
\varpi(\eta)\, \in \, {\rm Conn}(J^2(T{\mathbb X}))
$$
be the holomorphic connection on $J^2(T{\mathbb X})$ given by $\eta$ in Lemma
\ref{lem3}. Then $\varpi(\eta)$ is an oper connection.
\end{lemma}

\begin{proof}
Take a projective structure $\mathcal P$ on $\mathbb X$. Let
$$\delta({\mathcal P})\, \in\, H^0({\mathbb X},\, {\rm Diff}^3_{\mathbb X}(T{\mathbb X},\,
K^{\otimes 2}_{\mathbb X}))_0$$ be the differential operator corresponding to $\mathcal P$
in Proposition \ref{prop1}(2). Let
$$
{\mathbb D}(\mathcal P)\, \in\, {\rm Conn}(J^2(T{\mathbb X}))
$$
be the connection in Proposition \ref{prop1}(1) corresponding to $\mathcal P$.
It can be shown that
\begin{itemize}
\item $\varpi(\delta({\mathcal P}))\,=\,{\mathbb D}(\mathcal P)$, and

\item ${\mathbb D}(\mathcal P)$ is an oper connection.
\end{itemize}
Indeed, from the proof of Proposition \ref{prop1} we know that it suffices to prove this
for the (unique) standard projective structure on ${\mathbb P}({\mathbb V})$. Now, both these statements
are straight-forward for the unique projective structure on ${\mathbb P}({\mathbb V})$.

Next note that $H^0({\mathbb X},\, {\rm Diff}^3_{\mathbb X}(T{\mathbb X},\, K^{\otimes 
2}_{\mathbb X}))_0$ is an affine space for the complex vector space $H^0({\mathbb X},\, 
\text{Hom}(J^2(T{\mathbb X}),\, K^{\otimes 2}_{\mathbb X}))$. So $\eta$ in the statement of the 
lemma is of the form
\begin{equation}\label{the}
\eta\,=\, \delta({\mathcal P})+\theta\, , 
\end{equation}
where $\theta\, \in\, H^0({\mathbb X},\, \text{Hom}(J^2(T{\mathbb X}),\, 
K^{\otimes 2}_{\mathbb X}))$. The composition of homomorphisms
$$
J^2(T{\mathbb X})\, \stackrel{\theta}{\longrightarrow}\, K^{\otimes 2}_{\mathbb X} \,
\stackrel{l}{\longrightarrow}\, J^{2}(T{\mathbb X}) \otimes K_{\mathbb X}\, ,
$$
where $l$ is the homomorphism in \eqref{e4}, will be denoted by $\widetilde{\theta}$. From
\eqref{the} we have
\begin{equation}\label{the2} 
\varpi(\eta) - \varpi(\delta({\mathcal P}))\,=\, \varpi(\eta)- {\mathbb D}(\mathcal P)\,=\, 
\widetilde{\theta}\, .
\end{equation}
Since ${\mathbb D}(\mathcal P)$ is an oper connection, from 
\eqref{the2} it is straight-forward to deduce that $\varpi(\eta)$ is also an oper connection. 
\end{proof}

Take any ${\mathcal D}\, \in\, {\rm Conn}(J^2(T{\mathbb X}))$. Using $\mathcal D$ we shall 
construct an endomorphism of the vector bundle $J^{2}(T{\mathbb X})$. For that, let 
\begin{equation}\label{e5}
p_0\, :\, J^{2}(T{\mathbb X})\, \longrightarrow\, T{\mathbb X} 
\end{equation}
be the composition
$J^{2}(T{\mathbb X})\, \longrightarrow\, J^{1}(T{\mathbb X}) \, 
\longrightarrow\, T{\mathbb X}$
of the projections in the two short exact sequences in Remark 
\ref{rem-j}. For any $x\, \in\, \mathbb X$, and $v\, \in\, J^{2}(T{\mathbb X})_x$, let 
$\widetilde{v}$ be the unique section of $J^{2}(T{\mathbb X})$ defined on a simply connected open 
neighborhood of $x$ such that \begin{itemize} \item $\widetilde{v}$ is flat for the connection 
$\mathcal D$, and

\item $\widetilde{v}(x)\,=\, v$.
\end{itemize}
Now we have a holomorphic homomorphism
\begin{equation}\label{fd}
F_{\mathcal D}\, :\, J^{2}(T{\mathbb X})\, \longrightarrow\, J^{2}(T{\mathbb X})
\end{equation}
that sends any $v\, \in\, J^{2}(T{\mathbb X})_x$, $x\, \in\, \mathbb X$, to the element
of $J^{2}(T{\mathbb X})_x$ defined by the section $p_0(\widetilde{v})$,
where $p_0$ is the projection in \eqref{e5}, and $\widetilde v$ is constructed
as above using $v$ and $\mathcal D$.

\begin{lemma}\label{lem4}
A holomorphic connection ${\mathcal D}\,\in\,{\rm Conn}(J^2(T{\mathbb X}))$ lies in
${\rm image}(\varpi)$ (see Lemma \ref{lem3}) if and only if
\begin{itemize}
\item $\mathcal D$ is an oper connection, and

\item $F_{\mathcal D}\,=\, {\rm Id}_{J^{2}(T{\mathbb X})}$, where $F_{\mathcal D}$
is constructed in \eqref{fd}.
\end{itemize}
\end{lemma}

\begin{proof}
As in the proof of Lemma \ref{lemm0}, first take a projective structure $\mathcal P$ on
$\mathbb X$. Let $\delta({\mathcal P})$ (respectively, ${\mathbb D}(\mathcal P)$) be
the differential operator (respectively, holomorphic connection) corresponding to $\mathcal P$
as in the proof of Lemma \ref{lemm0}. We saw that $\varpi(\delta(\mathcal P))\,=\,
{\mathbb D}({\mathcal P})$ and ${\mathbb D}({\mathcal P})$ is an oper connection.
It can be shown that $F_{{\mathbb D}({\mathcal P})}\,=\,
{\rm Id}_{J^{2}(T{\mathbb X})}$, where $F_{{\mathbb D}({\mathcal P})}$
is constructed as in \eqref{fd}. Indeed, it suffices to prove this for
the unique projective structure on ${\mathbb P}({\mathbb V})$, which is actually straight-forward.

Now take $\eta\,=\, \delta({\mathcal P})+\theta$ as in \eqref{the}. Since
$\varpi(\delta(\mathcal P))\,=\,{\mathbb D}({\mathcal P})$ is an oper connection, and
$F_{{\mathbb D}({\mathcal P})}\,=\,{\rm Id}_{J^{2}(T{\mathbb X})}$, from \eqref{the2}
it follows that $F_{\varpi (\eta)}\,=\, {\rm Id}_{J^{2}(T{\mathbb X})}$, where $F_{\varpi(\eta)}$
is constructed as in \eqref{fd} for the connection $\varpi(\eta)$; it was shown in Lemma \ref{lemm0} that
$\varpi(\eta)$ is an oper connection.

To prove the converse, take any ${\mathcal D}\,\in\,{\rm Conn}(J^2(T{\mathbb X}))$. Then
\begin{equation}\label{nb}
{\mathcal D}\,=\, {\mathbb D}(\mathcal P)+\beta\, ,
\end{equation}
where $\beta\, \in\, H^0({\mathbb X},\, {\rm End}(J^2(T{\mathbb X}))\otimes K_{\mathbb X})$.

Now assume that
\begin{itemize}
\item $\mathcal D$ is an oper connection, and

\item $F_{\mathcal D}\,=\, {\rm Id}_{J^{2}(T{\mathbb X})}$, where $F_{\mathcal D}$
is constructed in \eqref{fd}.
\end{itemize}
Since ${\mathbb D}(\mathcal P)$ also satisfies these two conditions, it follows that
there is a unique section
$$\widetilde{\beta}\, \in\, H^0({\mathbb X},\, \text{Hom}(J^2(T{\mathbb X}),\, 
K^{\otimes 2}_{\mathbb X}))
$$
such that $\beta$ in \eqref{nb} coincides with the composition of homomorphisms
$$
J^2(T{\mathbb X})\, \stackrel{\widetilde\beta}{\longrightarrow}\, K^{\otimes 2}_{\mathbb X} \,
\stackrel{l}{\longrightarrow}\, J^{2}(T{\mathbb X}) \otimes K_{\mathbb X}\, ,
$$
where $l$ is the homomorphism in \eqref{e4}. Consequently, we have
$$
\varpi(\delta({\mathcal P})+\widetilde{\beta})\,=\, \mathcal D\, .
$$
This proves the lemma.
\end{proof}

\subsection{Differential operator given by projective structures}\label{sec3.2}

Given a projective structure $\mathcal P$ on a Riemann surface $\mathbb X$, recall that in Proposition 
\ref{prop1}(2) we constructed an element of $H^0({\mathbb X},\, \text{Diff}^3_{\mathbb 
X}(T{\mathbb X},\, K^{\otimes 2}_{\mathbb X}))_0$.

\begin{proposition}\label{prop2}
The space of all projective structures on $\mathbb X$ is in a natural bijection with the subspace
of $H^0({\mathbb X},\, {\rm Diff}^3_{\mathbb X}(T{\mathbb X},\,
K^{\otimes 2}_{\mathbb X}))_0$ consisting all differential operators $\delta$ satisfying
the following two conditions:
\begin{enumerate}
\item The connection on $\bigwedge^3 J^{2}(T{\mathbb X})$ induced by the connection 
$\varpi(\delta)$ on $J^2(T{\mathbb X})$ (see Lemma \ref{lem3}) coincides with the trivial 
connection on $\bigwedge^3 J^{2}(T{\mathbb X})$ (see Remark \ref{rem-j}).

\item If $s$ and $t$ are locally defined holomorphic sections of $T{\mathbb X}$ such that
$\delta(s)\,=\, 0\, =\, \delta(t)$, then $\delta([s,\, t])\,=\, 0$, where $[s,\, t]$ is the
usual Lie bracket of the vector fields $s$ and $t$.
\end{enumerate}
\end{proposition}

\begin{proof}
Let $\textbf{P}({\mathbb X})$ denote the space of all projective structures on $\mathbb X$. Let
$$\textbf{D}({\mathbb X})\, \subset\, H^0({\mathbb X},\, {\rm Diff}^3_{\mathbb X}(T{\mathbb X},\,
K^{\otimes 2}_{\mathbb X}))_0
$$
be the subset consisting of all differential operators $\delta$ satisfying the following conditions:
\begin{enumerate}
\item The connection on $\bigwedge^3 J^{2}(T{\mathbb X})$ induced by the connection 
$\varpi(\delta)$ on $J^2(T{\mathbb X})$ coincides with the trivial 
connection on $\bigwedge^3 J^{2}(T{\mathbb X})$.

\item If $s$ and $t$ are locally defined holomorphic sections of $T{\mathbb X}$ such that
$\delta(s)\,=\, 0\, =\, \delta(t)$, then $\delta([s,\, t])\,=\, 0$.
\end{enumerate}

Let ${\mathcal P}\,\in\, \textbf{P}({\mathbb X})$ be a projective structure on $\mathbb X$. Let
$$
\delta\, \in\, H^0({\mathbb X},\, {\rm Diff}^3_{\mathbb X}(T{\mathbb X},\,
K^{\otimes 2}_{\mathbb X}))_0
$$
be the differential operator given by $\mathcal P$ (see Proposition \ref{prop1}(2)).
In view of Proposition \ref{prop1}(4), the first one
of the above two conditions on $\delta$ is satisfied; also, the second 
condition is satisfied because of Proposition \ref{prop1}(3). Therefore, we get a map
\begin{equation}\label{Th}
\Theta\, :\, \textbf{P}({\mathbb X})\, \longrightarrow\,\textbf{D}({\mathbb X})
\end{equation}
that sends any $\mathcal P$ to the corresponding differential operator $\delta$.

There is a natural map
\begin{equation}\label{Psi}
\Psi\, :\, \textbf{D}({\mathbb X})\, \longrightarrow\,\textbf{P}({\mathbb X})
\end{equation}
(see \cite[p.~14, (3.7)]{Bi}). To clarify, set $n\,=\,2$ in the definition of $\mathcal B$ in
\cite[p.~13]{Bi}. Then ${\mathcal B}_0$ in \cite[p.~13]{Bi} coincides with the subset of
$H^0({\mathbb X},\, {\rm Diff}^3_{\mathbb X}(T{\mathbb X},\,
K^{\otimes 2}_{\mathbb X}))$ consisting of all $\delta'$ such that
\begin{itemize}
\item the symbol of $\delta'$ is the constant function $1$, and

\item the holomorphic connection on $\bigwedge^3 J^{2}(T{\mathbb X})$ induced by
the connection $\varpi(\delta')$ on $J^{2}(T{\mathbb X})$
coincides with the trivial connection on $\bigwedge^3 J^{2}(T{\mathbb X})\,=\, {\mathcal O}_{\mathbb X}$.
\end{itemize}

For the maps $\Theta$ and $\Psi$ constructed in \eqref{Th} and \eqref{Psi} respectively, we have
\begin{equation}\label{e6}
\Psi\circ\Theta\,=\, \text{Id}_{\textbf{P}({\mathbb X})}\, ;
\end{equation}
this follows from the combination of the facts that
\begin{itemize}
\item the map $F$ in \cite[p.~19, (5.4)]{Bi} is a bijection,

\item the map $\Psi$ in \eqref{Psi} coincides with the composition of the map $F$ in
\cite[p.~19, (5.4)]{Bi} with the natural projection $\textbf{P}({\mathbb X})\times
H^0({\mathbb X} , \, K^{\otimes 3}_{\mathbb X})\, \longrightarrow\,
H^0({\mathbb X} , \, K^{\otimes 3}_{\mathbb X})$, and

\item $F^{-1}({\mathcal P},\, 0)\,=\, \Theta ({\mathcal P})$ for all
${\mathcal P}\, \in\, \textbf{P}({\mathbb X})$, where $\Theta$ is the map in \eqref{Th}.
\end{itemize}

{}From \eqref{e6} we conclude that the map $\Theta$ in \eqref{Th} is injective.
We will prove that the map $\Theta$ is surjective as well.

Let
$$
H^0({\mathbb X},\, {\rm Diff}^3_{\mathbb X}(T{\mathbb X},\,
K^{\otimes 2}_{\mathbb X}))_1 \, \subset\,
H^0({\mathbb X},\, {\rm Diff}^3_{\mathbb X}(T{\mathbb X},\,
K^{\otimes 2}_{\mathbb X}))_0
$$
be the subset consisting of all
$\eta\, \in\, H^0({\mathbb X},\, {\rm Diff}^3_{\mathbb X}(T{\mathbb X},\,
K^{\otimes 2}_{\mathbb X}))_0$
such that the connection on $\bigwedge^3 J^{2}(T{\mathbb X})$ induced by the connection 
$\varpi(\eta)$ on $J^2(T{\mathbb X})$ (see Lemma \ref{lem3}) coincides with the trivial 
connection on $\bigwedge^3 J^{2}(T{\mathbb X})$ (see Remark \ref{rem-j}). Let
\begin{equation}\label{Psip}
\Psi'\, :\, H^0({\mathbb X},\, {\rm Diff}^3_{\mathbb X}(T{\mathbb X},\,
K^{\otimes 2}_{\mathbb X}))_1 \, \longrightarrow\,\textbf{P}({\mathbb X})
\end{equation}
be the map in \cite[p.~14, (3.7)]{Bi}; recall that the map $\Psi$ in \eqref{Psi}
is the restriction of this map $\Psi'$ to the subset $\textbf{D}({\mathbb X})$ of
$H^0({\mathbb X},\, {\rm Diff}^3_{\mathbb X}(T{\mathbb X},\,
K^{\otimes 2}_{\mathbb X}))_1$.

Take any
\begin{equation}\label{eeta}
\eta\, \in\, H^0({\mathbb X},\, {\rm Diff}^3_{\mathbb X}(T{\mathbb X},\,
K^{\otimes 2}_{\mathbb X}))_1\, .
\end{equation}
{}From the
isomorphism $F$ in \cite[p.~19, (5.4)]{Bi} we know that there is a holomorphic section 
$$
\xi\, \in\, H^0({\mathbb X},\, K^{\otimes 3}_{\mathbb X})
$$
such that
\begin{equation}\label{xi}
\eta\,=\, \Theta(\Psi'(\eta))+\xi\, ,
\end{equation}
where $\Psi'$ and $\Theta$ are the maps in \eqref{Psip} and \eqref{Th} respectively.

We now impose the following condition on $\eta$ in \eqref{eeta}:

If $s$ and $t$ are locally defined holomorphic sections of $T{\mathbb X}$ such that
$\eta(s)\,=\, 0\, =\, \eta(t)$, then $\eta([s,\, t])\,=\, 0$.

Since $\Theta(\Psi'(\eta))\, \in\, \textbf{D}({\mathbb X})$, in particular, 
$\Theta(\Psi'(\eta))$ satisfies the above condition, from the above
condition on $\eta$ it follows that the section $\xi$ in \eqref{xi} vanishes
identically. So, we have $\eta\,=\, \Theta(\Psi'(\eta))$. This implies that
the map $\Theta$ is surjective.
\end{proof}

As before, ${\rm Conn}(J^2(T{\mathbb X}))$ denotes the space of all holomorphic
connections on $J^2(T{\mathbb X})$.
Lemma \ref{lem4} and Proposition \ref{prop2} combine together to give the following:

\begin{corollary}\label{cor1}
The space of all projective structures on $\mathbb X$ is in a natural bijection with the subset
of ${\rm Conn}(J^2(T{\mathbb X}))$ defined by all connections
$\mathcal D$ satisfying the following four conditions:
\begin{enumerate}
\item $\mathcal D$ is an oper connection,

\item $F_{\mathcal D}\,=\, {\rm Id}_{J^{2}(T{\mathbb X})}$, where $F_{\mathcal D}$
is constructed in \eqref{fd},

\item the connection on $\bigwedge^3 J^{2}(T{\mathbb X})$ induced by $\mathcal D$, coincides with
the trivial connection on $\bigwedge^3 J^{2}(T{\mathbb X})$, and

\item if $s$ and $t$ are locally defined holomorphic sections of $T{\mathbb X}$ such that
the sections of $J^{2}(T{\mathbb X})$ corresponding to $s$ and $t$ are flat with respect to
${\mathcal D}$, then the section of $J^{2}(T{\mathbb X})$ corresponding to $[s,\, t]$ is also
flat with respect to ${\mathcal D}$.
\end{enumerate}
\end{corollary}

\begin{proof}
In Proposition \ref{prop1} we constructed a map from the projective structures on $\mathbb X$ to 
the holomorphic connections on $J^{2}(T{\mathbb X})$. The holomorphic connections on 
$J^{2}(T{\mathbb X})$ obtained this way satisfy all the four conditions in the statement of the 
corollary. See Proposition \ref{prop1} for conditions (3) and (4); see the proof
of Lemma \ref{lemm0} for condition (1); see the proof of Lemma \ref{lem4} for (2).

Conversely, let $\mathcal D$ be a holomorphic connection on $J^{2}(T{\mathbb X})$ 
satisfying the four conditions. In view of Lemma \ref{lem4}, from the first two conditions we 
conclude that ${\mathcal D}\,=\, \varpi(\delta)$
for some $\delta\, \in\, H^0({\mathbb 
X},\, {\rm Diff}^3_{\mathbb X}(T{\mathbb X},\, K^{\otimes 2}_{\mathbb X}))_0$.
In view of Proposition \ref{prop2}, from the third and fourth conditions we conclude that 
$\mathcal D$ corresponds to a projective structures on $\mathbb X$.
\end{proof}

\section{Projective structures and orthogonal opers}\label{sec4}

Projective structures on a Riemann surface $\mathbb X$ are precisely the $\text{PSL}(2,{\mathbb C})$--opers
on $\mathbb X$. On the other hand, we have the isomorphism $\text{PSL}(2,{\mathbb C})\,=\, \text{SO}(3,{\mathbb C})$. This
isomorphism is obtained by identifying ${\mathbb C}^3$ equipped with the standard nondegenerate symmetric form
and $\text{Sym}^2({\mathbb C}^2)$ equipped with the nondegenerate symmetric form constructed using the
standard anti-symmetric form on ${\mathbb C}^2$; this identification produces a homomorphism from
$\text{SL}(2,{\mathbb C})$ to $\text{SO}(3,{\mathbb C})$, which factors through $\text{PSL}(2,{\mathbb C})$,
producing an isomorphism of $\text{PSL}(2,{\mathbb C})$ with $\text{SO}(3,{\mathbb C})$.
Therefore, projective structures on $\mathbb X$ are
precisely the $\text{SO}(3,{\mathbb C})$--opers on $\mathbb X$. In this subsection we shall elaborate this
point of view.

A holomorphic $\text{SO}(3,{\mathbb C})$--bundle on a Riemann surface $\mathbb X$ consists of a holomorphic vector
bundle $W$ of rank three on $\mathbb X$ together with a holomorphic section $B_W\, \in\, H^0({\mathbb X},\,
\text{Sym}^2(W^*))$, such that
\begin{itemize}
\item $\bigwedge^3 W$ is identified with ${\mathcal O}_{\mathbb X}$ by a given isomorphism,

\item $B_W$ is fiberwise nondegenerate, and

\item the given isomorphism of $\bigwedge^3 W$ with ${\mathcal O}_{\mathbb X}$ takes the
bilinear form on $\bigwedge^3 W$ induced by $B_W$ to the standard constant bilinear form on
${\mathcal O}_{\mathbb X}$ (the corresponding quadratic form takes the section of ${\mathcal O}_{\mathbb X}$
given by the constant function $1$ on $\mathbb X$ to the function $1$).
\end{itemize}

A \textit{filtered} $\text{SO}(3,{\mathbb C})$--bundle on $\mathbb X$ is a holomorphic
$\text{SO}(3,{\mathbb C})$--bundle $(W,\, B_W)$ together with a filtration of holomorphic subbundles
$$
F^W_1\, \subset\, F^W_2\, \subset\, W
$$
such that
\begin{enumerate}
\item $F^W_1$ is holomorphically identified with $K_{\mathbb X}$ by a given isomorphism,

\item $B_W(F^W_1\otimes F^W_1)\, =\, 0$,

\item $F^W_2/F^W_1$ is holomorphically identified with ${\mathcal O}_{\mathbb X}$ by a given isomorphism,

\item $B_W(F^W_1\otimes F^W_2)\, =\, 0$; in other words, $(F^W_1)^\perp\, =\, F^W_2$.
\end{enumerate}

Note that the first and third conditions together imply that $W/F^W_2$ is holomorphically identified with
$\bigwedge^3 W \otimes (K_{\mathbb X})^* \,=\, T\mathbb X$.

A \textit{holomorphic connection} on a filtered $\text{SO}(3,{\mathbb C})$--bundle $(W,\, B_W,\, \{F^W_i\}_{i=1}^2)$
is a holomorphic connection $D_W$ on $W$ such that
\begin{itemize}
\item the holomorphic connection $D_W$ preserves the bilinear form $B_W$ on $W$,

\item the holomorphic connection on $\bigwedge^3 W\,=\, {\mathcal O}_{\mathbb X}$ induced by $D_W$ coincides
with the holomorphic
connection on ${\mathcal O}_{\mathbb X}$ given by the de Rham differential $d$,

\item $D_W(F^W_1)\,=\, F^W_2\otimes K_{\mathbb X}$, and

\item the second fundamental form of $F^W_1$ for $D_W$, which is a holomorphic section of
$\text{Hom}(K_{\mathbb X},\, {\mathcal O}_{\mathbb X})\otimes K_{\mathbb X}\,=\, {\mathcal O}_{\mathbb
X}$ (see \eqref{j1}), coincides with the section of ${\mathcal O}_{\mathbb
X}$ given by the constant function $1$ on $\mathbb X$.
\end{itemize}

An $\text{SO}(3,{\mathbb C})$--\textit{oper} on ${\mathbb X}$ is a filtered $\text{SO}(3,{\mathbb C})$--bundle 
$(W,\, B_W,\, \{F^W_i\}_{i=1}^2)$ equipped with a holomorphic connection $D_W$.

The above conditions imply that
the holomorphic section of $\text{Hom}({\mathcal O}_{\mathbb X},\, T{\mathbb X})\otimes K_{\mathbb X}\,=\,
{\mathcal O}_{\mathbb X}$ that gives the second fundamental form of $F^W_2$ for $D_W$ --- see \eqref{j2} --- coincides
with the one given by the constant function $1$. See \cite{BeDr} for more on
$\text{SO}(3,{\mathbb C})$--opers.

Recall from \eqref{eB} that $J^2(T{\mathbb P}(\mathbb V))$ is equipped with the bilinear form $B_0$, and
it has the filtration $\{F^{\mathbb P}_i\}_{i=1}^2$ constructed in \eqref{f1}. From Section \ref{se2.4} we conclude
that $$(J^2(T{\mathbb P}(\mathbb V)),\, B_0,\, \{F^{\mathbb P}_i\}_{i=1}^2,\, \mathbb{D}_0)$$ is an
$\text{SO}(3,{\mathbb C})$--oper on ${\mathbb P}(\mathbb V)$, where $\mathbb{D}_0$
is the holomorphic connection on $J^2(T{\mathbb P}(\mathbb V))$ constructed in \eqref{d}.
Consider the action of $\text{PGL}({\mathbb V})$ on $J^2(T{\mathbb P}(\mathbb V))$ given by the action
of $\text{PGL}({\mathbb V})$ on ${\mathbb P}(\mathbb V)$. Both $B_0$ and the filtration $\{F^{\mathbb P}_i\}_{i=1}^2$
are clearly preserved by this action of $\text{PGL}({\mathbb V})$ on $J^2(T{\mathbb P}(\mathbb V))$. As noted
in the proof of Proposition \ref{prop1}(1), the connection $\mathbb{D}_0$ is preserved by the
action of $\text{PGL}({\mathbb V})$ on $J^2(T{\mathbb P}(\mathbb V))$.

Let $(W,\, B_W,\, \{F^W_i\}_{i=1}^2,\, D_W)$ be an $\text{SO}(3,{\mathbb C})$--oper on ${\mathbb X}$.
Denote by $P(W)$ the projective bundle on $\mathbb X$ that parametrizes the lines in the fibers of $W$. Let
\begin{equation}\label{pw}
{\mathbf P}_W\, \subset\, P(W)
\end{equation}
be the ${\mathbb C}{\mathbb P}^1$--bundle on $\mathbb X$ that parametrizes the isotropic lines for $B_W$
(the lines on which the corresponding quadratic form vanishes). So the given condition
$B_W(F^W_1\otimes F^W_1)\, =\, 0$ implies that the line subbundle $F^W_1\, \subset\, W$
produces a holomorphic section
$$
s_W\, :\, {\mathbb X}\, \longrightarrow\, {\mathbf P}_W
$$
of the ${\mathbb C}{\mathbb P}^1$--bundle in \eqref{pw}.
The connection $D_W$ produces a holomorphic connection on $P(W)$, which, in turn,
induces a holomorphic connection on ${\mathbf P}_W$; this holomorphic
connection on ${\mathbf P}_W$ will be denoted by $\mathcal{H}_W$. It is straight-forward to check that
the triple $({\mathbf P}_W,\, \mathcal{H}_W,\, s_W)$ defines a projective structure on $\mathbb X$ (the
condition for the triple to define a projective structure is recalled in Section \ref{se2.1}).

Conversely, let $\mathcal P$ be a projective structure on $\mathbb X$. Take a
holomorphic coordinate atlas $\{(U_i,\, \phi_i)\}_{i\in I}$ in the equivalence class defining
$\mathcal P$. Using $\phi_i$,
\begin{itemize}
\item the bilinear form $B_0$ on $J^2(T{\mathbb P}(\mathbb V))\vert_{\phi_i(U_i)}$ constructed in
\eqref{eB} produces a bilinear form $B_{\mathcal P}(i)$ on $J^2(T{\mathbb X})\vert_{U_i}$,

\item the filtration $\{F^{\mathbb P}_j\}_{j=1}^2$ of $J^2(T{\mathbb P}(\mathbb V))\vert_{\phi_i(U_i)}$
constructed in \eqref{f1} produces a filtration $\{F^{\mathcal P}_j(i)\}_{j=1}^2$ of
$J^2(T{\mathbb X})\vert_{U_i}$, and

\item the holomorphic connection ${\mathbb D}_0\vert_{\phi_i(U_i)}$ in
\eqref{d} on $J^2(T{\mathbb P}(\mathbb V))\vert_{\phi_i(U_i)}$ produces a holomorphic connection
${\mathbb D}_{\mathcal P}(i)$ on $J^2(T{\mathbb X})\vert_{U_i}$.
\end{itemize}
Since $B_0$, $\{F^{\mathbb P}_j\}_{j=1}^2$ and ${\mathbb D}_0$ are all $\text{PGL}({\mathbb V})$--equivariant, each of
the locally defined structures $\{B_{\mathcal P}(i)\}_{i\in I}$, $\{\{F^{\mathcal P}_j(i)\}_{j=1}^2\}_{i\in I}$
and $\{{\mathbb D}_{\mathcal P}(i)\}_{i\in I}$ patch together compatibly to define
\begin{itemize}
\item a holomorphic nondegenerate symmetric bilinear form $B_{\mathcal P}$ on $J^2(T{\mathbb X})$,

\item a filtration $\{F^{\mathcal P}_j\}_{j=1}^2$ of holomorphic subbundles of $J^2(T{\mathbb X})$, and 

\item a holomorphic connection ${\mathbb D}_{\mathcal P}$ on $J^2(T{\mathbb X})$.
\end{itemize}
Since $(J^2(T{\mathbb P}(\mathbb V)),\, B_0,\, \{F^{\mathbb P}_i\}_{i=1}^2,\, \mathbb{D}_0)$ is an
$\text{SO}(3,{\mathbb C})$--oper on ${\mathbb P}(\mathbb V)$, we conclude that
$$(J^2(T{\mathbb X}),\, B_{\mathcal P},\, \{F^{\mathcal P}_j\}_{j=1}^2,\, {\mathbb D}_{\mathcal P})$$
is an $\text{SO}(3,{\mathbb C})$--oper on ${\mathbb X}$.

It is straight-forward to check that the above two constructions, namely from $\text{SO}(3,
{\mathbb C})$--opers on ${\mathbb X}$ to projective structures on ${\mathbb X}$ and vice versa, are inverses
of each other.

The above construction of an $\text{SO}(3,{\mathbb C})$--oper on ${\mathbb X}$ from $\mathcal P$ has the
following alternative description.

Let $\gamma\, :\, {\mathbf P}_{\mathcal P}\, \longrightarrow\, {\mathbb X}$ be a holomorphic ${\mathbb
C}{\mathbb P}^1$--bundle, ${\mathcal H}_{\mathcal P}$ a holomorphic connection on ${\mathbf P}_{\mathcal P}$
and $s_{\mathcal P}$ a holomorphic section of $\gamma$, such that the triple $$(\gamma,\, {\mathcal H}_{\mathcal P},
\, s_{\mathcal P})$$ gives the projective structure $\mathcal P$ (see Section \ref{se2.1}). Let
$$
W\, :=\, \gamma_* T_\gamma\, \longrightarrow\, {\mathbb X}
$$
be the direct image, where $T_\gamma$ is defined in \eqref{tga}. For any $x\, \in\, \mathbb X$, the
fiber $W_x$ is identified with $H^0(\gamma^{-1}(x), \, T(\gamma^{-1}(x)))$, so $W_x$ is a Lie
algebra isomorphic to $\mathfrak{sl}(2,{\mathbb C})$; the Lie algebra structure is given by the
Lie bracket of vector fields. Let $B_W$ denote the Killing form on $W$.

If $\mathbb V$ is a rank two holomorphic
vector bundle on $\mathbb X$ such that ${\mathbb P}({\mathbb V})\,=\, {\mathbf P}_{\mathcal P}$, then
$W\,=\, {\rm ad}({\mathbb V})\, \subset\, \text{End}({\mathbb V})$ (the subalgebra bundle of trace zero endomorphisms).
It should be clarified, that although $\mathbb V$ is not uniquely determined by ${\mathbf P}_{\mathcal P}$, any
two choices of $\mathbb V$ with ${\mathbb P}({\mathbb V})\,=\, {\mathbf P}_{\mathcal P}$ differ by tensoring
with a holomorphic line bundle on $\mathbb X$. As a consequence, $\text{End}({\mathbb V})$ and ${\rm ad}({\mathbb V})$
are uniquely determined by ${\mathbf P}_{\mathcal P}$.
The bilinear form $B_W$ coincides with the bilinear form on ${\rm ad}({\mathbb V})$ defined by the Killing bilinear form of the endomorphism algebra (defined from the trace).

Since the image $s_{\mathcal P}({\mathbb X})\, \subset\, {\mathbf P}_{\mathcal P}$
of the section $s_{\mathcal P}$ is a divisor, the
holomorphic vector bundle $W$ has a filtration of holomorphic subbundles
\begin{equation}\label{of2}
F^{\mathbb P}_1\, :=\,
\gamma_* (T_\gamma\otimes {\mathcal O}_{{\mathbf P}_{\mathcal P}}(-2s_{\mathcal P}({\mathbb X})))
\, \subset\, F^{\mathbb P}_2\, :=\,
\gamma_* (T_\gamma\otimes {\mathcal O}_{{\mathbf P}_{\mathcal P}}(-s_{\mathcal P}({\mathbb X})))
\, \subset\, \gamma_* T_\gamma\,=:\, W\, .
\end{equation}
Recall from Section \ref{se2.1} that the homomorphism $\widehat{ds_{\mathcal P}}\, :\,
T{\mathbb X}\, \longrightarrow\, (s_{\mathcal P})^*T_\gamma$ (see \eqref{dc2}) is an 
isomorphism. Therefore, $\widehat{ds_{\mathcal P}}$ gives a holomorphic isomorphism between $T\mathbb X$ and 
the normal bundle ${\mathbf N}\,=\, N_{s_{\mathcal P}({\mathbb X})}$ of the divisor $s_{\mathcal 
P}({\mathbb X})\,\subset\, {\mathbf P}_{\mathcal P}$. On the other hand, by the Poincar\'e 
adjunction formula, ${\mathbf N}^*$ is identified with the restriction ${\mathcal O}_{{\mathbf 
P}_{\mathcal P}}(-s_{\mathcal P}({\mathbb X}))\vert_{s_{\mathcal P}({\mathbb X})}$ of the 
holomorphic line bundle ${\mathcal O}_{{\mathbf P}_{\mathcal P}}(-s_{\mathcal P}({\mathbb X}))$ 
to the divisor $s_{\mathcal P}({\mathbb X})$; see \cite[p.~146]{GH} for the Poincar\'e adjunction 
formula. Therefore, the pulled back holomorphic line bundle
\begin{equation}\label{lb}
(s_{\mathcal P})^*(T_\gamma\otimes
{\mathcal O}_{{\mathbf P}_{\mathcal P}}(-2s_{\mathcal P}({\mathbb X})))
\end{equation}
is identified with $T\mathbb X\otimes K^{\otimes 2}_{\mathbb X}\,=\, K_{\mathbb X}$.

Since the line bundle in \eqref{lb} is canonically identified with $F^{\mathbb P}_1$ in \eqref{of2},
we conclude that $F^{\mathbb P}_1$ is identified with $K_{\mathbb X}$.

The quotient line bundle $F^{\mathbb P}_2/F^{\mathbb P}_1$ in \eqref{of2} is identified with
$$
(s_{\mathcal P})^*(T_\gamma\otimes
{\mathcal O}_{{\mathbf P}_{\mathcal P}}(-s_{\mathcal P}({\mathbb X})))\, .
$$
Since $(s_{\mathcal P})^*({\mathcal O}_{{\mathbf P}_{\mathcal P}}(-s_{\mathcal P}({\mathbb X})))$
is identified with $K_{\mathbb X}$, it follows that $F^{\mathbb P}_2/F^{\mathbb P}_1$ is
identified with ${\mathcal O}_{\mathbb X}$.

For any $x\, \in\, \mathbb X$, the subspace $(F^{\mathbb P}_1)_x\, \subset\, W_x$ is a nilpotent subalgebra, and
$(F^{\mathbb P}_2)_x\, \subset\, W_x$ is the unique Borel subalgebra containing $(F^{\mathbb P}_1)_x$.
Hence we have $B_W(F^{\mathbb P}_1\otimes F^{\mathbb P}_1)\,=\, 0$ and
$(F^{\mathbb P}_1)^\perp \,=\, F^{\mathbb P}_2$.

Consequently, $(W,\, B_W,\, \{F^{\mathbb P}_i\}_{i=1}^2)$ is a filtered $\text{SO}(3,{\mathbb 
C})$--bundle on $\mathbb X$. The holomorphic connection ${\mathcal H}_{\mathcal P}$ on ${\mathbf 
P}_{\mathcal P}$ produces a holomorphic connection
\begin{equation}\label{j-1}
{\mathbb D}_{\mathcal P}
\end{equation}
on $W$. Indeed,
from the fact that $\text{Aut}({\mathbb C}{\mathbb P}^1)\,=\, \text{PSL}(2, {\mathbb C})$ it follows
that the ${\mathbb C}{\mathbb P}^1$--bundle ${\mathbf P}_{\mathcal P}$ gives a holomorphic principal
$\text{PSL}(2, {\mathbb C})$--bundle ${\mathbf P}_{\mathcal P}(\text{PSL}(2, {\mathbb C}))$
on $\mathbb X$, and ${\mathcal H}_{\mathcal P}$ produces a
holomorphic connection on ${\mathbf P}_{\mathcal P}(\text{PSL}(2, {\mathbb C}))$. This
holomorphic connection on ${\mathbf P}_{\mathcal P}(\text{PSL}(2, {\mathbb C}))$ produces a
holomorphic connection on the vector bundle $W$ associated to ${\mathbf P}_{\mathcal P}(\text{PSL}(2, {\mathbb C}))$
for the adjoint action of $\text{PSL}(2, {\mathbb C})$ on its Lie algebra
$\mathfrak{sl}(2,{\mathbb C})$; this connection on $W$ is denoted by ${\mathbb D}_{\mathcal P}$.

The connection ${\mathbb D}_{\mathcal P}$ in \eqref{j-1} is in fact a 
holomorphic connection on the filtered $\text{SO}(3,{\mathbb C})$--bundle $(W,\, B_W,\, 
\{F^{\mathbb P}_i\}_{i=1}^2)$. The resulting $\text{SO}(3,{\mathbb C})$--oper $(W,\, B_W,\, \{F^{\mathbb 
P}_i\}_{i=1}^2,\, {\mathbb D}_{\mathcal P})$ coincides with the one constructed earlier from $\mathcal 
P$.

\section{Branched projective structures and logarithmic connections}\label{se5}

\subsection{Branched projective structure}\label{se5.1}

Let $X$ be a connected Riemann surface. Fix a nonempty finite subset
\begin{equation}\label{e7}
{\mathbb S}\, :=\, \{x_1, \, \cdots ,\, x_d\}\, \subset\, X
\end{equation}
of $d$ distinct points. For each $x_i$, $1\, \leq\, i\, \leq\, d$, fix an integer $n_i\, \geq\, 1$.
Let
\begin{equation}\label{e0n}
S\, :=\, \sum_{i=1}^d n_i\cdot x_i
\end{equation}
be the effective divisor on $X$.

The group of all holomorphic automorphisms of ${\mathbb C}{\mathbb P}^1$ is the M\"obius group
$\text{PGL}(2,{\mathbb C})$. Any
$$
\begin{pmatrix}
a & b\\
c& d
\end{pmatrix} \, \in\, \text{SL}(2,{\mathbb C})
$$
acts on ${\mathbb C}{\mathbb P}^1\,=\, {\mathbb C}\cup\{\infty\}$ as
$z\, \longmapsto\, \frac{az+b}{cz+d}$; the center ${\mathbb Z}/2\mathbb Z$ of
$\text{SL}(2,{\mathbb C})$ acts trivially,
thus producing an action of $\text{PGL}(2,{\mathbb C})\,=\, \text{SL}(2,{\mathbb C})/({\mathbb Z}
/2\mathbb Z)$ on ${\mathbb C}{\mathbb P}^1$.

A branched projective structure on $X$ with branching type $S$ (defined in \eqref{e0n})
is given by data
\begin{equation}\label{de1}
\{(U_j,\, \phi_j)\}_{j\in J}\, ,
\end{equation}
where
\begin{enumerate}
\item $U_j\, \subset\, X$ is a connected open subset with $\# (U_j\bigcap S_0)\, \leq\, 1$ such that
$\bigcup_{j\in J} U_j\,=\, X$,

\item $\phi_j\, :\, U_j\,\longrightarrow\, {\mathbb C}{\mathbb P}^1$ is a holomorphic map
which is an immersion on the complement $U_j\setminus (U_j\bigcap S_0)$,

\item if $U_j\bigcap S_0\,=\, x_i$, then $\phi_j$ is of degree $n_i+1$ and totally ramified at $x_i$,
while $\phi_j$ is an embedding if $U_j\bigcap S_0\,=\, \emptyset$, and

\item for every $j,\, j'\, \in\, J$ with $U_j\bigcap U_{j'}\,\not=\, \emptyset$, and
every connected component $U\, \subset\, U_j\bigcap U_{j'}$, there is an
element $f^U_{j,j'}\, \in \, \text{PGL}(2,{\mathbb C})$, such that $\phi_{j}\, =\, f^U_{j,j'}\circ
\phi_{j'}$ on $U$.
\end{enumerate}

Two such data $\{(U_j,\, \phi_j)\}_{j\in J}$ and $\{(U'_j,\, \phi'_j)\}_{j\in J'}$ satisfying the above
conditions are called \textit{equivalent} if their union $\{(U_j,\, \phi_j)\}_{j\in J}
\bigcup \{(U'_j,\, \phi'_j)\}_{j\in J'}$ also satisfies the above
conditions.

A \textit{branched projective structure} on $X$ with branching type $S$ is an equivalence
class of data $\{(U_j,\, \phi_j)\}_{j\in J}$ satisfying the above conditions. This definition
was introduced in \cite{Ma1}, \cite{Ma2} (see also \cite{BD} for the more general notion of a branched Cartan geometry).

We will now describe an equivalent formulation of the definition of a branched projective structure.

Consider a triple $(\gamma,\, {\mathcal H},\, s)$, where
\begin{itemize}
\item $\gamma\, :\, {\mathbf P}\, \longrightarrow\, \mathbb X$ is a holomorphic ${\mathbb C}{\mathbb P}^1$--bundle,

\item ${\mathcal H}\, \subset\, T{\mathbf P}$ is a holomorphic connection (as before,
$T{\mathbf P}$ is the holomorphic tangent bundle and ${\mathcal H}\oplus \text{kernel}(d\gamma)
\,=\, T{\mathbf P}$), and

\item $s\, :\, {\mathbb X}\, \longrightarrow\,{\mathbf P}$ is a holomorphic section of
$\gamma$,
\end{itemize}
such that the divisor for the homomorphism $\widehat{ds}$ in \eqref{dc2} coincides with $S$
in \eqref{e0n}.
This triple $(\gamma,\, {\mathcal H},\, s)$ gives a branched projective structure on $\mathbb X$
with branching type $S$.

Two such triples $(\gamma_1,\, {\mathcal H}_1,\, s_1)$ and $(\gamma_2,\, {\mathcal H}_2,\, s_2)$ are
called equivalent if there is a holomorphic isomorphism
$$
\mathbb{I}\, :\, {\mathbf P}_1\, \longrightarrow\, {\mathbf P}_2
$$
such that
\begin{itemize}
\item $\gamma_1\,=\, \gamma_2\circ\mathbb{I}$,

\item $d\mathbb{I} ({\mathcal H}_1)\,=\, \mathbb{I}^*{\mathcal H}_2$, where $d\mathbb{I}\, :\,
T{\mathbf P}_1\, \longrightarrow\, \mathbb{I}^*T{\mathbf P}_2$ is the differential of the map
$\mathbb{I}$, and

\item $\mathbb{I}\circ s_1\,=\, s_2$.
\end{itemize}

Two equivalent triples produce the same branched projective structure on $\mathbb X$. More precisely,
this map from the equivalence classes of triples to the 
branched projective structures on $\mathbb X$ with branching type $S$ is both injective and surjective.
More details on this can be found in \cite[Section 2.1]{BDG}.

For the convenience of the exposition, we would assume that we are in the generic situation where the
divisor $S$ in \eqref{e0n} is reduced. In other words, $n_i\,=\, 1$ for all $1\, \leq\, i\,\leq\, d$,
and 
\begin{equation}\label{a1}
S\, =\, \sum_{i=1}^d x_i\, .
\end{equation}

\subsection{Logarithmic connections}\label{se5.2}

Let $Y$ be a connected Riemann surface and ${\mathcal S}\, =\, \sum_{i=1}^m y_i$ be
an effective divisor on $Y$, where $y_1,\, \cdots,\, y_m$ are $m$ distinct points of $Y$.
The holomorphic cotangent bundle of $Y$ will
be denoted by $K_Y$. For any point $y\, \in\, \mathcal S$, the
fiber $(K_Y\otimes {\mathcal O}_Y({\mathcal S}))_y$ is identified with $\mathbb C$ by sending any
meromorphic $1$-form defined around
$y$ to its residue at $y$. In other words, for any holomorphic coordinate function
$z$ on $Y$ defined around the point $y$, with $z(y)\,=\, 0$, consider the isomorphism
\begin{equation}\label{ry}
R_y\, :\, (K_Y\otimes {\mathcal O}_Y({\mathcal S}))_y\, \longrightarrow\, {\mathbb C}\, ,
\ \ c\cdot \frac{dz}{z} \, \longmapsto\, c\, .
\end{equation}
This isomorphism is in fact independent of the choice of the above coordinate function $z$.

Let $W$ be a holomorphic vector bundle on $Y$.

A \textit{logarithmic connection} on $W$ singular over $\mathcal S$ is a holomorphic
differential operator of order one
$$
D\, :\, W\, \longrightarrow\, W\otimes K_Y\otimes {\mathcal O}_Y({\mathcal S})
$$
such that $D(fs) \,=\, f D(s) + s\otimes df$ for all locally defined holomorphic function $f$
on $Y$ and all locally defined holomorphic section $s$ of $W$. This means that
$$
D\, \in\, H^0(Y,\,\text{Diff}^1_Y(W,\, W\otimes K_Y\otimes {\mathcal O}_Y({\mathcal S})))
$$
such that the symbol of $D$
is the holomorphic section of $\text{End}(W)\otimes{\mathcal O}_Y({\mathcal S})$ given
by $\text{Id}_W\otimes 1$, where $1$ is the constant function $1$ on $Y$.

Note that a logarithmic connection $D$ defines a holomorphic connection on the restriction of $W$ 
to the complement $Y\setminus \mathcal S$. The logarithmic connection $D$ is called an extension
of this holomorphic connection on $W\vert_{Y\setminus S}$.

For a logarithmic connection $D_0$ on $W$ singular over $\mathcal S$, and a point $y\, \in\,
{\mathcal S}$, consider the composition of homomorphisms
$$
W\, \stackrel{D_0}{\longrightarrow}\, W\otimes K_Y\otimes {\mathcal O}_Y({\mathcal S})
\, \stackrel{\text{Id}_W\otimes
R_y}{\longrightarrow}\, W_y\otimes{\mathbb C}\,=\, W_y\, ,
$$
where $R_y$ is the homomorphism in \eqref{ry}. This composition of homomorphism vanishes
on the subsheaf $W\otimes {\mathcal O}_Y(-y)\, \subset\, W$, and hence it produces a homomorphism
$$
\text{Res}(D_0,y)\, :\, W/(W\otimes {\mathcal O}_Y(-y))\,=\, W_y\, \longrightarrow\, W_y\, .
$$
This endomorphism $\text{Res}(D_0,y)$ of $W_y$ is called the \textit{residue} of the
logarithmic connection $D_0$ at the point $y$; see \cite[p.~53]{De}.

The notion of second fundamental form in Section \ref{se2.4} extends to the set-up of logarithmic
connections.

Let $W$ be a holomorphic vector bundle on $Y$ equipped with a logarithmic connection
$D$ singular over $\mathcal S$. Let $W'\,\subset\, W$ be a holomorphic subbundle.
Then the composition of homomorphisms
$$
W\, \stackrel{D}{\longrightarrow}\, W\otimes K_Y\otimes {\mathcal O}_Y({\mathcal S})\,
\stackrel{q_{W'}\otimes{\rm Id}}{\longrightarrow}\,
(W/W')\otimes K_Y\otimes{\mathcal O}_Y({\mathcal S})\, ,
$$
where $q_{W'}\, :\, W\,\longrightarrow\,W/W'$ is the natural quotient map, defines a holomorphic 
section of $\text{Hom}(W',\, (W/W'))\otimes K_Y\otimes{\mathcal O}_Y({\mathcal S})$, which is called the
\textit{second fundamental form} of 
$W'$ for $D$. If $W''\, \subset\, W$ is a holomorphic subbundle containing $W'$ such that
$D(W')\, \subset\, W''\otimes K_Y\otimes{\mathcal O}_Y({\mathcal S})$, then the second 
fundamental form of $W'$ for $D$ is given by a section
\begin{equation}\label{j1n}
\zeta_1\, \in\,
H^0(Y,\, \text{Hom}(W',\, W''/W')\otimes K_Y\otimes{\mathcal O}_Y({\mathcal S}))
\end{equation}
using the natural inclusion map
$$
H^0(Y,\,\text{Hom}(W',\,W''/W')\otimes K_Y\otimes{\mathcal O}_Y({\mathcal S}))\,\hookrightarrow
\, H^0(Y,\,\text{Hom}(W',\,W/W')\otimes K_Y\otimes{\mathcal O}_Y({\mathcal S}))\, .$$
The second fundamental form of this subbundle $W''$ for $D$ is given by a section
\begin{equation}\label{j2n}
\zeta_2\, \in\,
H^0(Y,\,\text{Hom}(W''/W',\, W/W'')\otimes K_Y\otimes{\mathcal O}_Y({\mathcal S}))
\end{equation}
through the natural inclusion map 
$$H^0(Y,\,\text{Hom}(W''/W',\, W/W'')\otimes K_Y\otimes{\mathcal O}_Y({\mathcal S}))\, \hookrightarrow\,
H^0(Y,\,\text{Hom}(W'',\, W/W'')\otimes K_Y\otimes{\mathcal O}_Y({\mathcal S}))\, .$$

\subsection{Logarithmic connection from a branched projective structure}\label{5.3}

We now use the notation of Section \ref{se5.1}.

Let
\begin{equation}\label{e8}
{\mathbb T}\, :=\, (TX)\otimes {\mathcal O}_X(S)
\end{equation}
be the holomorphic line bundle on $X$, where $S$ is the divisor in \eqref{a1}. From the isomorphism
in \eqref{ry} it follows that the fiber ${\mathcal O}_X(S)_y$ is identified with $T_yX$ for every
$y\,\in\, \mathbb S$ in \eqref{e7}. For any point $y\, \in\, \mathbb S$, let
\begin{equation}\label{e9}
F^1_y\,:=\, (TX)\otimes {\mathcal O}_X(S)\otimes (T^*X)^{\otimes 2})_y\,=\,
{\mathbb C}\, \subset\, F^2_y\, \subset\, J^2({\mathbb T})_y
\end{equation}
be the filtration of subspaces of the fiber $J^2({\mathbb T})_y$ over $y$ constructed as in \eqref{ef};
more precisely, $F^1_y$ is the kernel of the projection $\alpha_y\, :\, J^2({\mathbb T})_y\,
\longrightarrow\, J^1({\mathbb T})_y$ (see \eqref{e1}), and $$F^2_y\,=\, \alpha^{-1}_y({\mathbb T}_y\otimes
(K_X)_y) \,=\, \alpha^{-1}_y({\mathcal O}(S)_y)$$
(see \eqref{e1}). In other words, $F^2_y$ is the kernel of the composition of homomorphisms
$$
J^2({\mathbb T})_y\, \stackrel{\alpha_y}{\longrightarrow}\, J^1({\mathbb T})_y\,
\longrightarrow\, {\mathbb T}_y
$$
(see \eqref{e1}). Note that $(TX)\otimes{\mathcal O}_X(S)\otimes (T^*X)^{\otimes 2})_y\,=\,
{\mathbb C}$, because ${\mathcal O}_X(S)_y\,=\, T_yX$.

The complement $X\setminus {\mathbb S}$ will be denoted by $\mathbb X$, where $\mathbb S$
is the subset in \eqref{e7}.

Let $P$ be a branched projective structure on $X$ of branching type $S$, where $S$ is the divisor 
in \eqref{a1}. So $P$ gives a projective structure on $\mathbb X\,:=\, X\setminus {\mathbb S}$; 
this projective structure on $\mathbb X$ will be denoted by $\mathcal P$. As shown in Proposition 
\ref{prop1}(1), the projective structure $\mathcal P$ produces a holomorphic connection ${\mathbb 
D} ({\mathcal P})$ on the vector bundle $J^2(T{\mathbb X})$.

\begin{proposition}\label{prop3}
The above holomorphic connection ${\mathbb D} ({\mathcal P})$ on $J^2(T{\mathbb X})$ extends to
a logarithmic connection ${\mathbb D} (P)$ on $J^2({\mathbb T})$, where $\mathbb T$ is
the holomorphic line bundle in \eqref{e8}.

For any $x_i\, \in\, \mathbb S$ in \eqref{e7}, the eigenvalues of the residue
${\rm Res}({\mathbb D} (P),x_i)$ are $\{-2,\, -1,\, 0\}$.

The eigenspace for the eigenvalue $-2$ of ${\rm Res}({\mathbb D} (P),x_i)$
is the line $F^1_{x_i}$ in \eqref{e9}. The eigenspace for the eigenvalue $-1$ of
${\rm Res}({\mathbb D} (P),x_i)$ is contained in the subspace $F^2_{x_i}$ in \eqref{e9}.
\end{proposition}

\begin{proof}
This proposition follows from some general properties of logarithmic connections which
we shall explain first.

Let $Y$ be a Riemann surface and $y_0\, \in\, Y$ a point; let $\iota\, :\, y_0\,
\hookrightarrow \, Y$ be the inclusion map. Take a holomorphic vector bundle ${\mathcal W}$ on $Y$, and
let ${\mathcal W}_{y_0}\, \longrightarrow\, Q\, \longrightarrow\, 0$ be a quotient of the fiber of ${\mathcal W}$
over the point $y_0$. Let
\begin{equation}\label{ed}
0\, \longrightarrow\, V\, \stackrel{\beta}{\longrightarrow}\,{\mathcal W} \, \longrightarrow\,
{\mathcal Q}\, :=\, i_*Q\,\longrightarrow\, 0
\end{equation}
be a short exact sequence of coherent analytic sheaves on $Y$;
so the sheaf $\mathcal Q$ is supported on the reduced point $y_0$. Let
\begin{equation}\label{ed1}
0\, \longrightarrow\, \text{kernel}(\beta(y_0))\, \longrightarrow\, V_{y_0}
\, \stackrel{\beta(y_0)}{\longrightarrow}\, {\mathcal W}_{y_0} \, \longrightarrow\,
\text{cokernel}(\beta(y_0))\,=\, Q \, \longrightarrow\, 0
\end{equation}
be the exact sequence of vector spaces obtained by restricting the exact sequence in
\eqref{ed} to the point $y_0$. It can be shown that there is a canonical isomorphism
\begin{equation}\label{ed2}
\text{kernel}(\beta(y_0))\, \stackrel{\sim}{\longrightarrow}\, Q\otimes (T_{y_0}Y)^*\, .
\end{equation}
To prove this, take any $v\, \in\, \text{kernel}(\beta(y_0))$, and let $\widetilde v$ be
a holomorphic section of $V$ defined around $y_0$ such that ${\widetilde v}(y_0)\,=\, v$.
The locally defined section $\beta({\widetilde v})$ of ${\mathcal W}$ vanishes at $y_0$, so its $1$-jet at $y_0$
gives an element of ${\mathcal W}_{y_0}\otimes (T_{y_0}Y)^*$; this element of
${\mathcal W}_{y_0}\otimes (T_{y_0}Y)^*$ will be denoted by $v_1$. Let $v'\, \in\, Q\otimes (T_{y_0}Y)^*$
denote the image of $v_1$ under the homomorphism $\beta(y_0)\times \text{Id}_{(T_{y_0}Y)^*}$. It is
straightforward to check that $v'$ does not depend on the choice of the above extension $\widetilde v$
of $v$. Consequently, we get a homomorphism as in \eqref{ed2}. This homomorphism is in fact an isomorphism.

Let $\nabla^{\mathcal W}$ be a logarithmic connection on ${\mathcal W}$ singular over $y_0$. Then
$\nabla^{\mathcal W}$ induces
a logarithmic connection on the subsheaf $V$ in \eqref{ed} if and only if the residue
$\text{Res}(\nabla^{\mathcal W},y_0)\, \in\, \text{End}({\mathcal W}_{y_0})$ preserves the subspace
$\beta(y_0)(V_{y_0})\, \subset\, {\mathcal W}_{y_0}$ in \eqref{ed1}.

Now assume that $\nabla^{\mathcal W}$ induces
a logarithmic connection $\nabla^1$ on $V$. Then $\text{Res}(\nabla^1,y_0)$ preserves the
subspace $\text{kernel}(\beta(y_0)) \, \subset\, V_{y_0}$ in \eqref{ed1}, and the
endomorphism of $$V_{y_0}/\text{kernel}(\beta(y_0))\,=\, \beta(y_0)(V_{y_0})$$ induced
by $\text{Res}(\nabla^1,y_0)$ coincides with the restriction of
$\text{Res}(\nabla^{\mathcal W},y_0)$ to $\beta(y_0)(V_{y_0})$. Let $\text{Res}(\nabla^{\mathcal W},y_0)_Q\ \in\,
\text{End}(Q)$ be the endomorphism induced by $\text{Res}(\nabla^{\mathcal W},y_0)$.
The restriction of
$\text{Res}(\nabla^1,y_0)$ to $\text{kernel}(\beta(y_0)) \, \subset\, V_{y_0}$ coincides
with $\text{Id}+\text{Res}(\nabla^{\mathcal W},y_0)_Q$; from \eqref{ed2} it follows that
$$\text{End}(\text{kernel}(\beta(y_0)))\,=\, \text{End}(Q)\, ,$$
so $\text{Res}(\nabla^{\mathcal W},y_0)_Q$ gives an endomorphism of $\text{kernel}(\beta(y_0))$.

Conversely, let $\nabla^V$ be a logarithmic connection on $V$ singular over $y_0$. Then $\nabla^V$ induces
a logarithmic connection on the holomorphic vector bundle $\mathcal W$ in \eqref{ed} if and only if the residue
$\text{Res}(\nabla^V,y_0)\, \in\, \text{End}(V_{y_0})$ preserves the subspace
$\text{kernel}(\beta(y_0)) \, \subset\, V_{y_0}$ in \eqref{ed1}.

Now assume that $\nabla^V$ induces a logarithmic connection $\nabla'$ on ${\mathcal W}$. Then $\nabla'$
gives the logarithmic connection $\nabla^V$ on $V$. Consequently, the residues of
$\nabla^V$ and $\nabla'$ are related in the fashion described above.

Take a point $x_i\, \in\, {\mathbb S}$ (see \eqref{e7}). Let $x_i\, \in\, U\, \subset\, X$ be a
sufficiently small contractible open neighborhood of $x_i$ in $X$, in particular,
$U\bigcap {\mathbb S}\,=\, x_i$. Let
$$
{\mathbb D}_1\, :=\, \{z\, \in\, {\mathbb C}\, \mid\, |z| \, <\, 1\}\, \subset\, \mathbb C
$$
be the unit disk. Fix a biholomorphism
$$
\widetilde{\gamma}\, :\, {\mathbb D}_1\, \longrightarrow\, U
$$
such that $\widetilde{\gamma}(0)\,=\, x_i$. Let
\begin{equation}\label{e10}
\gamma\, :\, U\, \longrightarrow\, {\mathbb D}_1\, ,\ \ x\, \longmapsto\,
(\widetilde{\gamma}^{-1}(x))^2
\end{equation}
be the branched covering map. Then using $\gamma$, the branched projective structure on $U$, given by the
branched projective structure $P$ on $X$ of branching type $S$, produces an usual (unbranched) projective
structure on ${\mathbb D}_1$; this projective structure on ${\mathbb D}_1$ will be denoted by $P_1$.

Now substituting $({\mathbb D}_1,\, P_1)$ in place of $({\mathbb X},\, {\mathcal P})$ in 
Proposition \ref{prop1}(1), we get a holomorphic connection ${\mathbb D}(P_1)$ on $J^2(T{\mathbb 
D}_1)$. The holomorphic vector
bundle $\gamma^*J^2(T{\mathbb D}_1)$ over $U$, where $\gamma$ is the map in \eqref{e10},
is equipped with the holomorphic connection $\gamma^*{\mathbb D}(P_1)$.

The differential $d\gamma\, :\, TU\, \longrightarrow\, \gamma^*T{\mathbb D}_1$ induces
an isomorphism ${\mathbb T}\vert_U\, \stackrel{\sim}{\longrightarrow}\, \gamma^*T{\mathbb D}_1$, where
${\mathbb T}$ is the line bundle in \eqref{e8}. This, in turn, produces an isomorphism
\begin{equation}\label{is}
J^2({\mathbb T})\vert_U\, \stackrel{\sim}{\longrightarrow}\, J^2(\gamma^*T{\mathbb D}_1)\, .
\end{equation}
On the other hand, there is a natural homomorphism $\gamma^*J^2(T{\mathbb D}_1)
\, \longrightarrow\, J^2(\gamma^*T{\mathbb D}_1)$. Combining this with the inverse of the
isomorphism in \eqref{is} we get a homomorphism
\begin{equation}\label{p1}
\gamma^*J^2(T{\mathbb D}_1) \, \longrightarrow\, J^2({\mathbb T})\vert_U\, .
\end{equation}
This homomorphism is an isomorphism over $U\setminus\{x_i\}$.

Apply the above mentioned criterion, for the extension of a logarithmic connection on
the vector bundle $V$ in \eqref{ed} to a logarithmic connection on
the vector bundle ${\mathcal W}$, to the homomorphism in \eqref{p1} and
the holomorphic connection
$\gamma^*{\mathbb D}(P_1)$ on $\gamma^*J^2(T{\mathbb D}_1)$. We conclude that
$\gamma^*{\mathbb D}(P_1)$ induces a logarithmic connection on
$J^2({\mathbb T})\vert_U$.
The first statement of the proposition follows from this.

The other two statements of the proposition follow from the earlier mentioned properties
of residues of the logarithmic connections on the vector bundles $V$ and $\mathcal W$ in \eqref{ed} that
are induced by each other.
\end{proof}

\section{Branched projective structures and branched ${\rm SO}(3,{\mathbb C})$-opers}\label{sec6}

\subsection{Branched ${\rm SO}(3,{\mathbb C})$-opers}\label{se6.1}

We shall now use the terminology in Section \ref{sec4}.

\begin{definition}\label{def2}
Let $(W,\, B_W)$ be a holomorphic $\text{SO}(3,{\mathbb C})$--bundle on a connected Riemann surface $X$.
A \textit{branched filtration} on $(W,\, B_W)$, of type $S$ (see \eqref{a1}),
is a filtration of holomorphic subbundles
\begin{equation}\label{of}
F^W_1\, \subset\, F^W_2\, \subset\, W
\end{equation}
such that
\begin{itemize}
\item $F^W_1$ is holomorphically identified with $K_X\otimes{\mathcal O}_X(-S)$
by a given isomorphism,

\item $B_W(F^W_1\otimes F^W_1)\, =\, 0$,

\item $F^W_2/F^W_1$ is holomorphically identified with ${\mathcal O}_X$ by a given isomorphism,

\item $B_W(F^W_1\otimes F^W_2)\, =\, 0$ (equivalently, $(F^W_1)^\perp\, =\, F^W_2$).
\end{itemize}
\end{definition}

Compare the above definition with the definition of a filtration of $(W,\, B_W)$ given in Section \ref{sec4}.
They differ only at the first condition.

The above conditions imply that
\begin{equation}\label{w2}
W/F^W_2\,=\, (\bigwedge\nolimits ^3 W) \otimes (K_X\otimes{\mathcal O}_X(-S))^* \,=\,
TX\otimes{\mathcal O}_X(S)\, .
\end{equation}

A \textit{branched filtered} $\text{SO}(3,{\mathbb C})$--bundle, of type $S$, is a holomorphic $\text{SO}(3,
{\mathbb C})$--bundle $(W,\, B_W)$ equipped with a branched filtration $\{F^W_i\}_{i=1}^2$ of type $S$ as in \eqref{of}.

\begin{definition}\label{def3}
A \textit{branched holomorphic connection} on a branched filtered $\text{SO}(3,{\mathbb C})$--bundle
$(W,\, B_W,\, \{F^W_i\}_{i=1}^2)$, of type $S$, is a holomorphic connection $D_W$ on $W$ such that
\begin{itemize}
\item the holomorphic connection $D_W$ preserves the bilinear form $B_W$ on $W$,

\item the holomorphic connection on $\bigwedge^3 W\,=\, {\mathcal O}_X$ induced by $D_W$ coincides
with the holomorphic connection on ${\mathcal O}_X$ given by the de Rham differential $d$,

\item $D_W(F^W_1)\,\subset\, F^W_2\otimes K_X$, and

\item the second fundamental form of $F^W_1$ for $D_W$, which is a holomorphic section of
$$\text{Hom}(K_X\otimes{\mathcal O}_X(-S),\, {\mathcal O}_X)\otimes K_X
\,=\, {\mathcal O}_X(S)
$$ (see \eqref{j1}), coincides with the section of ${\mathcal O}_X(S)$ given by the constant function
$1$ on $X$.
\end{itemize}
\end{definition}

\begin{definition}\label{def4}
A \textit{branched} $\text{SO}(3,{\mathbb C})$--\textit{oper}, of type $S$, on $X$ is a
branched filtered $\text{SO}(3,{\mathbb C})$--bundle
$$(W,\, B_W,\, \{F^W_i\}_{i=1}^2)\, ,$$ of type $S$, equipped with a branched holomorphic connection $D_W$.
\end{definition}

Since the divisor $S$ is fixed, we would drop mentioning it explicitly.

The above conditions imply that
the holomorphic section of $$\text{Hom}({\mathcal O}_X,\, TX\otimes{\mathcal O}_X(S))\otimes K_X\,=\,
{\mathcal O}_X(S)$$ that gives the second fundamental form of $F^W_2$ for $D_W$ --- see \eqref{j2} --- coincides
with the section of ${\mathcal O}_X(S)$ given by the constant function $1$ on $X$.

\subsection{Branched projective structures are branched $\text{SO}(3,{\mathbb C})$-opers}

Take a branched $\text{SO}(3,{\mathbb C})$--oper $(W,\, B_W,\, \{F^W_i\}_{i=1}^2, \, D_W)$ on $X$ of type $S$.
Let $P(W)$ be the projective bundle on $X$ that parametrizes the lines in the fibers of $W$. As in
\eqref{pw},
$$
P(W)\, \supset\, {\mathbf P}_W\, \stackrel{\gamma}{\longrightarrow} \, X
$$
is the ${\mathbb C}{\mathbb P}^1$--bundle on $X$ that parametrizes the isotropic lines for $B_W$. 
The line subbundle $F^W_1\, \subset\, W$, being isotropic, produces a holomorphic section
\begin{equation}\label{sw}
s_W\, :\, X\, \longrightarrow\, {\mathbf P}_W
\end{equation}
of the projection $\gamma$.
The connection $D_W$ produces a holomorphic connection on $P(W)$, and this connection on $P(W)$
induces a holomorphic connection on ${\mathbf P}_W$; recall that a
holomorphic connection on ${\mathbf P}_W$ is a holomorphic line subbundle of $T{\mathbf P}_W$ transversal
to the relative tangent bundle $T_\gamma$ for the projection $\gamma$. Let
\begin{equation}\label{nec}
\mathcal{H}_W\, \subset\, T{\mathbf P}_W
\end{equation}
be the holomorphic connection on ${\mathbf P}_W$ given by $D_W$.

\begin{lemma}\label{lem5}
The triple $({\mathbf P}_W,\, \mathcal{H}_W,\, s_W)$ (see \eqref{sw} and \eqref{nec}) defines a
branched projective structure on $X$ of branching type $S$.
\end{lemma}

\begin{proof}
Recall that the last condition in the definition of a branched holomorphic connection on
$(W,\, B_W,\, \{F^W_i\}_{i=1}^2)$ (see Definition \ref{def3})
says that the second fundamental form of $F^W_1$ for $D_W$ is the
holomorphic section of
$$\text{Hom}(K_X\otimes{\mathcal O}_X(-S),\, {\mathcal O}_X)\otimes K_X
\,=\, {\mathcal O}_X(S)
$$
given by the constant function $1$ on $X$. On the other hand, the divisor for the second fundamental form of
$F^W_1$ for $D_W$ coincides with the divisor for the homomorphism $$\widehat{ds_W}\, :\, TX
\, \longrightarrow\, (s_W)^*T_\gamma$$ (see \eqref{dc2}). Consequently,
the divisor for the homomorphism $\widehat{ds_W}$ is $S$. Hence 
$({\mathbf P}_W,\, \mathcal{H}_W,\, s_W)$ defines a
branched projective structure on $X$ of branching type $S$; see Section \ref{se5.1}.
\end{proof}

Now let $\mathcal P$ be a branched projective structure on $X$ of branching type $S$.
Let $$\gamma\, :\, {\mathbf P}_{\mathcal P}\, \longrightarrow\, X$$ be a holomorphic ${\mathbb
C}{\mathbb P}^1$--bundle, ${\mathcal H}_{\mathcal P}$ a holomorphic connection on ${\mathbf P}_{\mathcal P}$
and $s_{\mathcal P}$ a holomorphic section of $\gamma$, such that the triple
\begin{equation}\label{tr}
(\gamma,\, {\mathcal H}_{\mathcal P}, \, s_{\mathcal P})
\end{equation}
gives the branched projective structure $\mathcal P$; see Section \ref{se5.1}. Define the
holomorphic vector bundle of rank three on $X$
\begin{equation}\label{ew}
W\, :=\, \gamma_* T_\gamma\, \longrightarrow\, X\, ,
\end{equation}
where $T_\gamma\, \subset\, T{\mathbf P}_{\mathcal P}$ is the relative holomorphic tangent 
for the projection $\gamma$ (as in \eqref{tga}). If $s$ and $t$ are locally defined holomorphic
sections of $T_\gamma\, \subset\, T{\mathbf P}_{\mathcal P}$, then the Lie bracket
$[s,\, t]$ is also a section of $T_\gamma$.
Consequently, each fiber of $W\, =\, \gamma_* T_\gamma$ is a Lie algebra
isomorphic to $\mathfrak{sl}(2,{\mathbb C})$. Let
\begin{equation}\label{ew1}
B_W\, \in\, H^0(X,\, \text{Sym}^2(W^*))
\end{equation}
be the holomorphic section given by the fiberwise Killing form on $W$.

As shown in Section \ref{sec4}, the pair $(W,\, B_W)$ can also be constructed by choosing a rank two holomorphic 
vector bundle $\mathbb V$ on $X$ such that ${\mathbb P}({\mathbb V})\,=\, {\mathbf P}_{\mathcal P}$; then $W\,=\, 
{\rm ad}({\mathbb V})\, \subset\, \text{End}({\mathbb V})$ (the subalgebra bundle of trace zero endomorphisms) and 
$B_W$ is constructed using the trace map of endomorphisms.

Exactly as in \eqref{of2}, construct the filtration
\begin{equation}\label{ot3}
F^{\mathbb P}_1\, :=\,
\gamma_* (T_\gamma\otimes {\mathcal O}_{{\mathbf P}_{\mathcal P}}(-2s_{\mathcal P}(X)))
\, \subset\, F^{\mathbb P}_2\, :=\,
\gamma_* (T_\gamma\otimes {\mathcal O}_{{\mathbf P}_{\mathcal P}}(-s_{\mathcal P}(X)))
\, \subset\, \gamma_* T_\gamma\,=:\, W
\end{equation}
of $W$.

\begin{lemma}\label{lem6}
The triple $(W,\, B_W,\, \{F^{\mathbb P}_i\}_{i=1}^2)$ constructed in \eqref{ew}, \eqref{ew1}, \eqref{ot3}
is a branched filtered ${\rm SO}(3,{\mathbb C})$--bundle of type $S$ (see \eqref{a1}).
\end{lemma}

\begin{proof}
The holomorphic line bundle $F^{\mathbb P}_1$ in \eqref{ot3} admits a canonical isomorphism
\begin{equation}\label{l1}
F^{\mathbb P}_1\, \stackrel{\sim}{\longrightarrow}\, (s_{\mathcal P})^*(T_\gamma\otimes
{\mathcal O}_{{\mathbf P}_{\mathcal P}}(-2s_{\mathcal P}(X)))\, ,
\end{equation}
because the restriction of $T_\gamma\otimes {\mathcal O}_{{\mathbf P}_{\mathcal P}}(-2s_{\mathcal P}(X))$
to any fiber of $\gamma$ is holomorphically trivializable, and the evaluation, at a given point, of
the global sections of a holomorphically trivializable bundle is an isomorphism.
The quotient bundle $F^{\mathbb P}_2/F^{\mathbb P}_1$ admits a canonical isomorphism
\begin{equation}\label{l2}
F^{\mathbb P}_2/F^{\mathbb P}_1\, \stackrel{\sim}{\longrightarrow}\,
(s_{\mathcal P})^*(T_\gamma\otimes
{\mathcal O}_{{\mathbf P}_{\mathcal P}}(-s_{\mathcal P}(X)))\, ,
\end{equation}
which is again constructed by evaluating the holomorphic sections of
$T_\gamma\otimes {\mathcal O}_{{\mathbf P}_{\mathcal P}}(-s_{\mathcal P}(X))\vert_{\gamma^{-1}(x)}$
at the point $s_{\mathcal P}(x)$ for every $x\, \in\, X$.

Now, by the Poincar\'e adjunction formula, $(s_{\mathcal P})^*
{\mathcal O}_{{\mathbf P}_{\mathcal P}}(-s_{\mathcal P}(X))$
is identified with the dual bundle $(s_{\mathcal P})^*{\mathbf N}^*$, where
${\mathbf N}\,=\, N_{s_{\mathcal P}(X)}$ is the normal bundle of
the divisor $s_{\mathcal P}(X)\, \subset\, {\mathbf P}_{\mathcal P}$. On the other hand,
$\mathbf N$ is canonically identified with the restriction
$T_\gamma\vert_{s_{\mathcal P}(X)}$, because 
the divisor $s_{\mathcal P}(X)$ is transversal to the fibration $\gamma$. Consequently,
$(s_{\mathcal P})^*{\mathcal O}_{{\mathbf P}_{\mathcal P}}(-s_{\mathcal P}(X))$
is identified with $(s_{\mathcal P})^*T^*_\gamma$.

Recall from Section \ref{se5.1} that the divisor for the homomorphism $\widehat{ds_{\mathcal P}}\, :\,
TX\, \longrightarrow\, (s_{\mathcal P})^* T_\gamma$ (see 
\eqref{dc2}) coincides with $S$. Consequently, $\widehat{ds_{\mathcal P}}$ identifies $(TX)
\otimes {\mathcal O}_X(S)$ with $(s_{\mathcal P})^*T_\gamma$. Therefore, the
line bundle $(s_{\mathcal P})^*(T_\gamma\otimes
{\mathcal O}_{{\mathbf P}_{\mathcal P}}(-2s_{\mathcal P}(X)))$
in \eqref{l1} is identified with $K_X\otimes {\mathcal O}_X(-S)$,
and the line bundle $(s_{\mathcal P})^*(T_\gamma\otimes
{\mathcal O}_{{\mathbf P}_{\mathcal P}}(-s_{\mathcal P}(X)))$
in \eqref{l2} is identified with ${\mathcal O}_X$.

{}From the above descriptions of $F^{\mathbb P}_1$ and $F^{\mathbb P}_2$ it follows that
for each point $x\, \in\, X$, the fiber $(F^{\mathbb P}_1)_x\, \subset\, W_x$ is a nilpotent
subalgebra of the Lie algebra $W_x$, and $(F^{\mathbb P}_2)_x\, \subset\, W_x$ is the unique
Borel subalgebra of $W_x$ containing $(F^{\mathbb P}_1)_x$. These imply that $B_W(F^{\mathbb
P}_1\otimes F^{\mathbb P}_1)\, =\, 0$, and $(F^{\mathbb P}_1)^\perp\,=\, F^{\mathbb P}_2$.
Hence $(W,\, B_W,\, \{F^{\mathbb P}_i\}_{i=1}^2)$ is 
branched filtered ${\rm SO}(3,{\mathbb C})$--bundle of type $S$.
\end{proof}

Consider the holomorphic connection ${\mathcal H}_{\mathcal P}$ in \eqref{tr} on 
the holomorphic ${\mathbb C}{\mathbb P}^1$--bundle ${\mathbf P}_{\mathcal P}$. It produces
a holomorphic connection on the direct image $W\, =\, \gamma_* T_\gamma$; see \eqref{j-1}. This
holomorphic connection on $W$ will be denoted by ${\mathbb D}_{\mathcal P}$.

\begin{lemma}\label{lem7}
The above connection ${\mathbb D}_{\mathcal P}$ on $W$ is a branched holomorphic connection
on the branched filtered ${\rm SO}(3,{\mathbb C})$--bundle $(W,\, B_W,\, \{F^{\mathbb P}_i\}_{i=1}^2)$
in Lemma \ref{lem6}. In other words, $$(W,\, B_W,\, \{F^{\mathbb P}_i\}_{i=1}^2,\, {\mathbb D}_{\mathcal P})$$
is a ${\rm SO}(3,{\mathbb C})$--oper.
\end{lemma}

\begin{proof}
The connection ${\mathbb D}_{\mathcal P}$ preserves the bilinear form $B_W$ on $W$,
because the action of $\text{PSL}(2,{\mathbb C})$ on $\text{Sym}^2({\mathbb C}^2)$ preserves
the symmetric bilinear form on $\text{Sym}^2({\mathbb C}^2)$ given by the standard
symplectic form on ${\mathbb C}^2$. Also, the holomorphic connection on $\bigwedge^3 W$ induced by
${\mathbb D}_{\mathcal P}$ coincides with the holomorphic
connection on ${\mathcal O}_X$ given by the de Rham differential $d$, because the action of
$\text{PSL}(2,{\mathbb C})$ on $\bigwedge^3 \text{Sym}^2({\mathbb C}^2)$ is the trivial action.

Since any holomorphic connection on a holomorphic bundle over a Riemann surface is automatically
integrable, the ${\mathbb C}{\mathbb P}^1$--bundle ${\mathbf P}_{\mathcal P}$ is locally isomorphic
the trivial holomorphic
${\mathbb C}{\mathbb P}^1$--bundle, and ${\mathcal H}_{\mathcal P}$ in \eqref{tr} is locally
holomorphically isomorphic to the trivial connection on the trivial holomorphic
${\mathbb C}{\mathbb P}^1$--bundle. So $W$ is locally
holomorphically isomorphic to the trivial vector bundle whose fibers are quadratic polynomials in
one variable, and ${\mathbb D}_{\mathcal P}$ is the trivial connection on this locally defined trivial
holomorphic vector bundle. With respect to these trivialization, and a suitable pair $(U,\, \phi)$ as
in \eqref{de1} compatible with the projective structure $\mathcal P$, the section $s_{\mathcal P}$ in
\eqref{tr} around any point $x_i\, \in\, \mathbb S$ is of the form $z\, \longmapsto (z,\, z^2)$, where
$z$ is a holomorphic function around $x_i$ with $z(x_i)\,=\, 0$.

In view of the above observations, from a straight-forward computation it follows that
\begin{itemize}
\item ${\mathbb D}_{\mathcal P}(F^{\mathbb P}_1)\,\subset\, F^{\mathbb P}_2\otimes K_X$, and

\item the second fundamental form of $F^{\mathbb P}_1$ for ${\mathbb D}_{\mathcal P}$, which is a
holomorphic section of
$$\text{Hom}(K_X\otimes{\mathcal O}_X(-S),\, {\mathcal O}_X)\otimes K_X
\,=\, {\mathcal O}_X(S)\, ,
$$
coincides with the section of ${\mathcal O}_X(S)$ given by the constant function $1$ on $X$.
\end{itemize}
This completes the proof.
\end{proof}

The above construction of a branched projective structure on $X$ from a 
branched ${\rm SO}(3,{\mathbb C})$--oper (see Lemma \ref{lem5}), and the construction of a
branched ${\rm SO}(3,{\mathbb C})$--oper from a branched projective structure (see Lemma \ref{lem7}),
are clearly inverses of each other.

We summarize the constructions done in this subsection in the following theorem.

\begin{theorem}\label{thm1}
There is a natural bijective correspondence between the branched projective structures on $X$
with branching type $S$ and the branched ${\rm SO}(3,{\mathbb C})$--opers on $X$ of type $S$.
\end{theorem}

\subsection{Logarithmic connection from branched $\text{SO}(3,{\mathbb C})$-opers}\label{se6.3}

In Proposition \ref{prop3} we constructed a logarithmic connection on $J^2({\mathbb T})$ from a branched
projective structure on $X$, where ${\mathbb T}\,=\,
(TX)\otimes {\mathcal O}_X(S)$ (see \eqref{e8}). On the other hand, Theorem
\ref{thm1} identifies branched projective structure on $X$ with branched ${\rm SO}(3,{\mathbb C})$--opers
on $X$. Thus a branched ${\rm SO}(3,{\mathbb C})$--oper gives a logarithmic connection on $J^2({\mathbb T})$.
Now we shall give a direct construction of the logarithmic connection on $J^2({\mathbb T})$
associated to a branched ${\rm SO}(3,{\mathbb C})$--oper on $X$.

Let
\begin{equation}\label{e11}
F^1_{\mathbb T}\, \subset\, F^2_{\mathbb T}\, \subset\, J^2({\mathbb T})
\end{equation}
be the filtration of holomorphic subbundles constructed as in \eqref{ef}; so,
$F^1_{\mathbb T}$ is the kernel of the natural projection
\begin{equation}\label{e14}
\textbf{c}_1\, :\, J^2({\mathbb T})\, \longrightarrow\, J^1({\mathbb T})
\end{equation}
(see \eqref{e1}), and $F^2_{\mathbb T}$
is the kernel of the homomorphism
\begin{equation}\label{e13}
\textbf{c}_2\, :\, J^2({\mathbb T})\, \longrightarrow\, {\mathbb T}
\end{equation}
obtained by composing $J^2({\mathbb T})\,\stackrel{\textbf{c}_1}{\longrightarrow}\,
J^1({\mathbb T})$ with the natural homomorphism $J^1({\mathbb T})\, \longrightarrow\, {\mathbb T}$
in \eqref{e1}. Note that the restriction of the filtration in \eqref{e11}
to any point $y\, \in\, \mathbb S$ coincides with the filtration in \eqref{e9}. We have
\begin{equation}\label{f1n}
F^1_{\mathbb T}\,=\, {\mathbb T}\otimes K^2_X\,=\, K_X\otimes{\mathcal O}_X(S)\, , \ \
F^2_{\mathbb T}/F^1_{\mathbb T} \,=\, {\mathbb T}\otimes K_X\,=\, {\mathcal O}_X(S)\, ,\ \
J^2({\mathbb T})/F^2_{\mathbb T}\,=\, {\mathbb T}\, .
\end{equation}

Let
\begin{equation}\label{tr2}
(W,\, B_W,\, \{F^W_i\}_{i=1}^2, \, D_W)
\end{equation}
be a branched $\text{SO}(3,{\mathbb C})$--oper on $X$ of type $S$. Recall from \eqref{w2} and \eqref{e8}
that $W/F^W_2\,=\, TX\otimes {\mathcal O}_X(S)\, =\, {\mathbb T}$.
Let
\begin{equation}\label{q0}
q_0\,:\, W \, \longrightarrow\, {\mathbb T}\,=\,W/F^W_2
\end{equation}
be the quotient map.

Take any point $x\, \in\, X$ and any $w\, \in\, W_x$. Let $\widetilde{w}$ be the unique holomorphic section of
$W$, defined on a simply connected open neighborhood of $x$, such that
$\widetilde{w}$ is flat with respect to the holomorphic connection $D_W$ in \eqref{tr2}, and
\begin{equation}\label{e12}
\widetilde{w}(x)\,=\, w\, .
\end{equation}
Let
\begin{equation}\label{Phi}
\Phi\, :\, W \, \longrightarrow\, J^2({\mathbb T})
\end{equation}
be the homomorphism that sends any $w\, \in\, W_x$, $x\, \in\, X$, to the element of
$J^2({\mathbb T})_x$ given by the restriction of the section $q_0(\widetilde{w})$ to the second order
infinitesimal neighborhood of $x$, where $q_0$ is the homomorphism in \eqref{q0} and $\widetilde{w}$
is constructed as above from $w$. This construction is similar to the construction of
the homomorphisms $\psi_j$ in \eqref{e0}.

\begin{proposition}\label{prop4}\mbox{}
\begin{enumerate}
\item For the homomorphism $\Phi$ in \eqref{Phi},
$$\Phi(F^W_1)\, \subset\, F^1_{\mathbb T}\ \ \ { and } \ \ \ \Phi(F^W_2)\, \subset\, F^2_{\mathbb T}\, $$ where
$\{F^i_{\mathbb T}\}_{i=1}^2$ and $\{F^W_i\}_{i=1}^2$ are the filtrations in \eqref{e11} and \eqref{tr2} respectively.

\item The homomorphism $\Phi$ takes the holomorphic connection $D_W$ in \eqref{tr2}
to a logarithmic connection on $J^2({\mathbb T})$ whose singular locus is $\mathbb S$ in \eqref{e7}.
The logarithmic connection on $J^2({\mathbb T})$ induced by $D_W$ will be denoted by ${\mathbb D}_J$.

\item For any $x_i\, \in\, \mathbb S$ (see \eqref{e7}), the residue ${\rm Res}({\mathbb D}_J, x_i)$
of ${\mathbb D}_J$ at $x_i$ has eigenvalues $\{-2,\, -1,\, 0\}$.

\item The eigenspace of ${\rm Res}({\mathbb D}_J, x_i)$ for the eigenvalue $-2$ is the line 
$(F^1_{\mathbb T})_{x_i}\, \subset\, J^2({\mathbb T})_{x_i}$ in \eqref{e11}.
The eigenspace of ${\rm Res}({\mathbb D}_J, x_i)$ for the eigenvalue $-1$ is contained in the
subspace $(F^2_{\mathbb T})_{x_i}\, \subset\, J^2({\mathbb T})_{x_i}$ in \eqref{e11}.
\end{enumerate}
\end{proposition}

\begin{proof}
To prove statement (1), first take any $x\, \in\, X$ and any $w\, \in\, (F^W_2)_x$.
Since $q_0(w)\,=\, 0$, where $q_0$ is the projection in \eqref{q0}, from \eqref{e12} it
follows immediately that $\textbf{c}_2\circ\Phi(w)\,=\, 0$, where $\textbf{c}_2$ is the homomorphism
in \eqref{e13}. This implies that $\Phi(F^W_2)\, \subset\, F^2_{\mathbb T}$. To prove statement (1),
we need to show that $\Phi(F^W_1)\, \subset\, F^1_{\mathbb T}$.

Take any $x\, \in\, X$ and any $w\, \in\, (F^W_1)_x$. From the given
condition that $D_W(F^W_1)\, \subset\, F^W_2\otimes K_X$ it follows that $\textbf{c}_1\circ\Phi(w)\,=\,
0$, where $\textbf{c}_1$ is the homomorphism in \eqref{e14}. More precisely, the restriction
$\Phi\vert_{F^W_1}$ coincides with the natural inclusion homomorphism
$$
F^W_1\,=\, K_X\otimes {\mathcal O}_X(-S)\, \hookrightarrow\, K_X\otimes {\mathcal O}_X(S)\,=\,
F^1_{\mathbb T}\, ;
$$
see \eqref{f1n} for $K_X\otimes{\mathcal O}_X(S)\,=\,
F^1_{\mathbb T}$ and Definition \ref{def2} for $F^W_1\,=\, K_X\otimes{\mathcal O}_X(-S)$.
Consequently, we have
$\Phi(F^W_1)\, \subset\, F^1_{\mathbb T}$. This proves (1).

Denote $X\setminus {\mathbb S}$ by $\mathbb X$.
We note that the restriction
$$
\Phi\vert_{\mathbb X}\, :\, W\vert_{\mathbb X} \, \longrightarrow\, J^2({\mathbb T})\vert_{\mathbb X}
\,=\, J^2(T{\mathbb X})
$$
is a holomorphic isomorphism.

{}From Theorem \ref{thm1} we know that the branched
$\text{SO}(3,{\mathbb C})$--oper $(W,\, B_W,\, \{F^W_i\}_{i=1}^2, \, D_W)$
in \eqref{tr2} defines a branched projective structure on $X$ of type $S$. Let $\mathcal P$ denote the
projective structure on ${\mathbb X}\,=\, X\setminus \mathbb S$ given by this
branched projective structure on $X$. Using Proposition
\ref{prop1}(1), the projective structure $\mathcal P$ yields a holomorphic connection ${\mathbb
D} ({\mathcal P})$ on $J^2(T{\mathbb X})$. This holomorphic connection ${\mathbb
D} ({\mathcal P})$ evidently coincides with the holomorphic connection ${\mathbb D}_J\vert_{\mathbb X}$
on $J^2({\mathbb T})\vert_{\mathbb X}\,=\, J^2(T{\mathbb X})$ given by the connection $D_W$
using the isomorphism $\Phi\vert_{\mathbb X}$. Consequently, the statements
(2), (3) and (4) in the proposition follow from Proposition \ref{prop3} and statement (1).
\end{proof}

Proposition \ref{prop4} yields the following:

\begin{corollary}\label{cor-1}
For the logarithmic connection ${\mathbb D}_J$ on $J^2({\mathbb T})$ in Proposition \ref{prop4}(2),
$$
{\mathbb D}_J(F^1_{\mathbb T})\, = \, F^2_{\mathbb T}\otimes K_X\otimes{\mathcal O}_X(S)\, ,
$$
and
$$
{\mathbb D}_J(F^2_{\mathbb T})\, =\, J^2({\mathbb T})\otimes K_X\otimes{\mathcal O}_X(S)\, ,
$$
where $\{F^i_{\mathbb T}\}_{i=1}^2$ is the filtration in \eqref{e11}.
\end{corollary}

\begin{proof}
{}From Proposition \ref{prop4}(2) we know that the logarithmic connection
${\mathbb D}_J$ is given by the connection $D_W$
using $\Phi$. Consequently, the corollary follows from Proposition \ref{prop4}(1) and the
properties, given in Proposition \ref{prop4}(3) and Proposition \ref{prop4}(4),
of the residue of the logarithmic connection ${\mathbb D}_J$.
\end{proof}

Recall the second fundamental form for a logarithmic connection
defined in Section \ref{se5.2}. Consider the logarithmic connection
${\mathbb D}_J$ on $J^2({\mathbb T})$ in Proposition \ref{prop4}(2)
and the subbundles $F^1_{\mathbb T},\, F^2_{\mathbb T}$ in \eqref{e11}.
Let
$$
\textbf{S}({\mathbb D}_J, F^1_{\mathbb T}) \ \ \ \text{ and }\ \ \ 
\textbf{S}({\mathbb D}_J, F^1_{\mathbb T})
$$
be the second fundamental forms of $F^1_{\mathbb T}$ and $F^2_{\mathbb T}$ respectively for ${\mathbb D}_J$.

{}From Corollary \ref{cor-1} and \eqref{j1n} we know that
\begin{equation}\label{h-1}
\textbf{S}({\mathbb D}_J, F^1_{\mathbb T})\, \in\, H^0(X,\, \text{Hom}(F^1_{\mathbb T},\,
F^2_{\mathbb T}/F^1_{\mathbb T})\otimes K_X\otimes{\mathcal O}_X(S))\,=\, H^0(X,\, {\mathcal O}_X(S))\, ;
\end{equation}
see \eqref{f1n}. From Corollary \ref{cor-1} and \eqref{j2n} we have
\begin{equation}\label{h-2}
\textbf{S}({\mathbb D}_J, F^2_{\mathbb T})\, \in\, H^0(X,\, \text{Hom}(F^2_{\mathbb T}/F^1_{\mathbb T},\,
J^2({\mathbb T})/F^2_{\mathbb T})\otimes K_X\otimes{\mathcal O}_X(S))\,=\, H^0(X,\, {\mathcal O}_X(S))\, ;
\end{equation}
see \eqref{f1n}. 

\begin{lemma}\label{lem8}
The second fundamental forms ${\bf S}({\mathbb D}_J, F^1_{\mathbb T})$ and
${\bf S}({\mathbb D}_J, F^2_{\mathbb T})$, in \eqref{h-1} and \eqref{h-2} respectively,
coincide with the section of ${\mathcal O}_X(S)$
given by the constant function $1$ on $X$.
\end{lemma}

\begin{proof}
As in the proof of Proposition \ref{prop4}, consider the branched projective structure on $X$ of type $S$
given by the $\text{SO}(3,{\mathbb C})$--oper $(W,\, B_W,\, \{F^W_i\}_{i=1}^2, \, D_W)$
in \eqref{tr2} using Theorem \ref{thm1}. It defines a projective structure on
${\mathbb X}\,=\, X\setminus \mathbb S$. Now from the statement (1) in Corollary \ref{cor1} we conclude
that the restrictions to ${\mathbb X}$ of both ${\bf S}({\mathbb D}_J, F^1_{\mathbb T})$ and
${\bf S}({\mathbb D}_J, F^2_{\mathbb T})$ coincide with the section of ${\mathcal O}_{\mathbb X}$
given by the constant function $1$ on ${\mathbb X}$. Hence the second fundamental forms
${\bf S}({\mathbb D}_J, F^1_{\mathbb T})$ and
${\bf S}({\mathbb D}_J, F^2_{\mathbb T})$ coincide with the section of ${\mathcal O}_X(S)$
given by the constant function $1$ on $X$, because ${\mathbb X}$ is a dense open subset of $X$.
\end{proof}

\subsection{A twisted symmetric form}\label{se6.4}

We continue with the set-up of Section \ref{se6.3}.

Using the homomorphism $\Phi$ in \eqref{Phi}, the nondegenerate symmetric form $B_W$ on
$W$ in \eqref{tr2} produces a nondegenerate symmetric form
\begin{equation}\label{bj}
\mathbb{B}_J\, \in\, H^0(X,\, \text{Sym}^2(J^2({\mathbb T})^*)\otimes{\mathcal O}_X(2S))\, .
\end{equation}
This follows from the fact that the image of the following composition of homomorphisms
$$
J^2({\mathbb T})^* \, \stackrel{\Phi^*}{\longrightarrow}\, W^*\, \stackrel{\sim}{\longrightarrow}\,
W \, \stackrel{\Phi}{\longrightarrow}\, J^2({\mathbb T})
$$
coincides with the subsheaf $J^2({\mathbb T})\otimes{\mathcal O}_X(-2S) \, \subset\,
J^2({\mathbb T})$; the above isomorphism
$W^*\, \stackrel{\sim}{\longrightarrow}\, W$ is given by the nondegenerate symmetric form $B_W$.

The logarithmic connection ${\mathbb D}_J$ on $J^2({\mathbb T})$ in Proposition \ref{prop4}(2), and the
canonical logarithmic connection on ${\mathcal O}_X(2S)$ defined by the de Rham differential, together
define a logarithmic connection on the vector bundle $\text{Sym}^2(J^2({\mathbb T})^*)
\otimes{\mathcal O}_X(2S)$. This logarithmic connection on $\text{Sym}^2(J^2({\mathbb T})^*)
\otimes{\mathcal O}_X(2S)$ will be denoted by $\text{Sym}^2({\mathbb D}_J)'$.

\begin{proposition}\label{prop5}\mbox{}
\begin{enumerate}
\item The form $\mathbb{B}_J$ in \eqref{bj} is covariant constant with respect to the above
logarithmic connection ${\rm Sym}^2({\mathbb D}_J)'$ on the vector bundle
${\rm Sym}^2(J^2({\mathbb T})^*)\otimes{\mathcal O}_X(2S)$.

\item For the subbundles $F^1_{\mathbb T},\, F^2_{\mathbb T}$ in \eqref{e11},
$$
\mathbb{B}_J(F^1_{\mathbb T}\otimes F^1_{\mathbb T})\,=\, 0\ \ { and }\ \
(F^1_{\mathbb T})^\perp \,=\, F^2_{\mathbb T}\, ,
$$
where $(F^1_{\mathbb T})^\perp\, \subset\, J^2({\mathbb T})$ is the subbundle orthogonal to
$F^1_{\mathbb T}$ with respect to $\mathbb{B}_J$.
\end{enumerate}
\end{proposition}

\begin{proof}
This proposition can be proved exactly as Lemma \ref{lem8} is proved. Indeed, from
the statement (1) in Corollary \ref{cor1} we know that both the statements in the proposition
holds over ${\mathbb X}\,=\, X\setminus \mathbb S$. Hence the proposition follows.
\end{proof}

Consider the logarithmic connection ${\mathbb D}_J$ in Proposition \ref{prop4}(2) and the twisted
holomorphic symmetric
bilinear form $\mathbb{B}_J$ in \eqref{bj} on the vector bundle
$J^2({\mathbb T})$ constructed from the branched $\text{SO}(3,{\mathbb C})$--oper
$(W,\, B_W,\, \{F^W_i\}_{i=1}^2, \, D_W)$ in \eqref{tr2}. We will now show that $(W,\, B_W,\, \{F^W_i\}_{i=1}^2,
\, D_W)$ can be reconstructed back from this pair
\begin{equation}\label{cp}
(\mathbb{B}_J,\, {\mathbb D}_J)\, .
\end{equation}

For each point $x_i \in \, \mathbb S$ (see \eqref{e7}), let $L_i\, \subset\, J^2({\mathbb T})_{x_i}$ be the
eigenspace for the eigenvalue $0$ of the residue ${\rm Res}({\mathbb D}_J, x_i)$. Let $\mathcal F$ be the
holomorphic vector bundle on $X$ that fits in the short exact sequence
\begin{equation}\label{g1}
0\, \longrightarrow\, {\mathcal F} \, \longrightarrow\, J^2({\mathbb T})\, \longrightarrow\,
\bigoplus_{i=1}^d J^2({\mathbb T})_{x_i}/L_i \, \longrightarrow\, 0
\end{equation}
of coherent analytic sheaves on $X$. Recall the criterion for a logarithmic connection on the vector
bundle $\mathcal W$ in \eqref{ed} to induce a logarithmic connection on the vector
bundle $V$ in \eqref{ed}. Applying this criterion to \eqref{g1}
it follows that ${\mathbb D}_J$ induces a logarithmic connection
${\mathbb D}'_J$ on $\mathcal F$. Moreover, the eigenvalues of the residue ${\rm Res}({\mathbb D}'_J, x_i)$
at $x_i\, \in\, {\mathbb S}$ are $(-1,\, 0,\, 0)$; this again follows from the expression of the residue,
of the logarithmic connection on the vector
bundle $V$ in \eqref{ed} induced by a logarithmic connection on $\mathcal W$, in terms of the
residue of the logarithmic connection on $\mathcal W$.

For each point $x_i \in \, \mathbb S$, let $M_i\, \subset\, {\mathcal F}_{x_i}$ be the
eigenspace for the eigenvalue $0$ of the residue ${\rm Res}({\mathbb D}'_J, x_i)$. Let $\mathcal E$ be the
holomorphic vector bundle on $X$ that fits in the short exact sequence
\begin{equation}\label{g2}
0\, \longrightarrow\, {\mathcal E} \, \longrightarrow\, {\mathcal F}\, \longrightarrow\,
\bigoplus_{i=1}^d {\mathcal F}_{x_i}/M_i \, \longrightarrow\, 0
\end{equation}
of coherent analytic sheaves on $X$. From the above mentioned criterion
it follows that the logarithmic connection ${\mathbb D}'_J$ induces a logarithmic connection
${\mathbb D}''_J$ on $\mathcal E$. Moreover, the eigenvalues of the residue ${\rm Res}({\mathbb D}''_J, x_i)$
at any $x_i\, \in\, {\mathbb S}$ are $(0,\, 0,\, 0)$. 

Using the homomorphism $\Phi$ in \eqref{Phi}, consider $W$ as a subsheaf of $J^2({\mathbb T})$.
On the other hand, from \eqref{g1} and \eqref{g2} we have ${\mathcal E}\, \subset\,
{\mathcal F}\, \subset\,J^2({\mathbb T})$, using which ${\mathcal E}$ will be considered as
a subsheaf of $J^2({\mathbb T})$.
It is straightforward to check that ${\mathcal E}\, \subset\, J^2({\mathbb T})$
coincides with the subsheaf $W$ of $J^2({\mathbb T})$. This identification between ${\mathcal E}$ and $W$ takes
${\mathbb D}''_J$ to the holomorphic connection $D_W$ in \eqref{tr2}. In particular,
the logarithmic connection ${\mathbb D}''_J$ is actually a nonsingular connection on ${\mathcal E}$.

The nondegenerate twisted symmetric form $\mathbb{B}_J$ on $J^2({\mathbb T})$ in \eqref{bj}
produces a twisted symmetric form on the subsheaf ${\mathcal E}\, \subset\, J^2({\mathbb T})$.
The above identification between ${\mathcal E}$ and $W$ takes this twisted symmetric form on $\mathcal E$
given by $\mathbb{B}_J$ to $B_W$ in \eqref{tr2}. In particular, the twisted symmetric form on $\mathcal E$
given by $\mathbb{B}_J$ is nondegenerate and there is no nontrivial twisting.

The filtration $\{F^W_i\}_{i=1}^2$ of $W\,=\, {\mathcal E}$ in \eqref{tr2}
is given by the filtration $\{F^i_{\mathbb T}\}_{i=1}^2$ of $J^2({\mathbb T})$ in \eqref{e11}. In other words,
$F^W_i$ is the unique holomorphic subbundle of ${\mathcal E}$ such that the space of holomorphic sections
of $F^W_i$ over any open subset $U\, \subset\, X$ is the space of holomorphic sections $s$ of ${\mathcal E}\vert_U$
such that $s\vert_{U\cap (X\setminus \mathbb S)}$ is a section of $F^i_{\mathbb T}$ over
$U\cap (X\setminus \mathbb S)$.

This way we recover the branched $\text{SO}(3,{\mathbb C})$--oper
$(W,\, B_W,\, \{F^W_i\}_{i=1}^2, \, D_W)$ in \eqref{tr2} from the pair $(\mathbb{B}_J,\, {\mathbb D}_J)$
in \eqref{cp} constructed from it.

\section{A characterization of branched $\text{SO}(3,{\mathbb C})$-opers}\label{sec7}

Let ${\mathfrak g}\,=\, \text{Lie}(\text{SO}(3,{\mathbb C}))$ be the Lie algebra of
$\text{SO}(3,{\mathbb C})$. We will need a property of ${\mathfrak g}$ which is formulated below.

Up to conjugacy, there is only one nonzero nilpotent element in ${\mathfrak g}$. Indeed,
this follows immediately from the fact that ${\mathfrak g}\,=\, \mathfrak{sl}(2, {\mathbb C})$.
Let
$$
A\, \in\, {\mathfrak g}\,=\, \text{Lie}(\text{SO}(3,{\mathbb C}))
$$
be a nonzero nilpotent element. From the above observation we know that $\dim A({\mathbb C}^3)\,=\, 2$.
Therefore, if $B\, \in\, {\mathfrak g}$ is a nilpotent element such that
$B(V_0)\,=\, 0$, where $V_0\, \subset\, {\mathbb C}^3$ is some subspace of dimension two, then
\begin{equation}\label{b1}
B\,=\, 0\, .
\end{equation}

Take a pair
\begin{equation}\label{tp}
({\mathbb B},\, {\mathbf D})\, ,
\end{equation}
where
\begin{itemize}
\item ${\mathbb B}\, \in\, H^0(X,\, \text{Sym}^2(J^2({\mathbb T})^*)\otimes{\mathcal O}_X(2S))$ is
fiberwise nondegenerate bilinear form on $J^2({\mathbb T})$ with values in $\mathcal{O}_X(2S)$, and

\item ${\mathbf D}$ is a logarithmic connection on $J^2({\mathbb T})$ singular over $S$,
\end{itemize}
such that the following five conditions hold:
\begin{enumerate}
\item For the subbundles $F^1_{\mathbb T},\, F^2_{\mathbb T}$ in \eqref{e11},
$$
\mathbb{B}(F^1_{\mathbb T}\otimes F^1_{\mathbb T})\,=\, 0\ \ \text{ and }\ \
(F^1_{\mathbb T})^\perp \,=\, F^2_{\mathbb T}\, ,
$$
where $(F^1_{\mathbb T})^\perp\, \subset\, J^2({\mathbb T})$ is the subbundle orthogonal to
$F^1_{\mathbb T}$ with respect to $\mathbb{B}$.

\item The section $\mathbb{B}$ is covariant constant with respect to the
logarithmic connection on the vector bundle ${\rm Sym}^2(J^2({\mathbb T})^*)\otimes{\mathcal O}_X(2S)$ induced by
${\mathbf D}$ and the logarithmic connection on $\mathcal{O}_X(2S)$ given by the de Rham differential $d$.

\item ${\mathbf D}(F^1_{\mathbb T})\, = \, F^2_{\mathbb T}\otimes K_X\otimes{\mathcal O}_X(S)$ and
${\mathbf D}(F^2_{\mathbb T})\, =\, J^2({\mathbb T})\otimes K_X\otimes{\mathcal O}_X(S)$.

\item For any $x_i\, \in\, \mathbb S$ (see \eqref{e7}), the residue ${\rm Res}({\mathbf D}, x_i)$
of ${\mathbf D}$ at $x_i$ has eigenvalues $\{-2,\, -1,\, 0\}$.

\item The eigenspace of ${\rm Res}({\mathbf D}, x_i)$ for the eigenvalue $-2$ is the line
$(F^1_{\mathbb T})_{x_i}\, \subset\, J^2({\mathbb T})_{x_i}$ in \eqref{e11}.
The eigenspace of ${\rm Res}({\mathbf D}, x_i)$ for the eigenvalue $-1$ is contained in the
subspace $(F^2_{\mathbb T})_{x_i}\, \subset\, J^2({\mathbb T})_{x_i}$.
\end{enumerate}

In other words, the pair $({\mathbb B},\, {\mathbf D})$ satisfies all properties obtained in
Proposition \ref{prop4}, Corollary \ref{cor-1}, Lemma \ref{lem8} and Proposition \ref{prop5} for
the pair in \eqref{cp} corresponding to the branched $\text{SO}(3,{\mathbb C})$--oper
$(W,\, B_W,\, \{F^W_i\}_{i=1}^2, \, D_W)$ in \eqref{tr2}. However this does not ensure that
$({\mathbb B},\, {\mathbf D})$ defines a branched $\text{SO}(3,{\mathbb C})$--oper. The reason for this
is that the logarithmic connection ${\mathbf D}$ might possess local monodromy around some points
of the subset $\mathbb S$ in \eqref{e7}. On the other hand, if the local monodromy of ${\mathbf D}$
around every point of $\mathbb S$ is trivial, then it can be shown that $({\mathbb B},\, {\mathbf D})$ defines
a branched $\text{SO}(3,{\mathbb C})$--oper (see Remark \ref{reml}).

Take a point $x_i\, \in \, \mathbb S$. Since the eigenvalues of ${\rm Res}({\mathbf D}, x_i)$
are $\{-2,\, -1,\, 0\}$ (the fourth one of the five conditions above), the local monodromy of ${\mathbf D}$ around
$x_i$ is unipotent (meaning $1$ is the only eigenvalue). Let
\begin{equation}\label{el}
L^i_2,\, L^i_1,\, L^i_0\, \subset\, J^2({\mathbb T})_{x_i}
\end{equation}
be the eigenspaces of ${\rm Res}({\mathbf D}, x_i)$ for the eigenvalues $-2,\, -1,\, 0$ respectively, so
\begin{equation}\label{el0}
L^i_2\oplus L^i_1\oplus L^i_0\, =\, J^2({\mathbb T})_{x_i}\, .
\end{equation}
Using
the logarithmic connection $\mathbf D$, we will construct a homomorphism
\begin{equation}\label{el2}
\varphi_i\, \in \, \text{Hom}(L^i_1\otimes (K_X)_{x_i},\, L^i_2\otimes (K^{\otimes 2}_X)_{x_i})
\,=\, \text{Hom}(L^i_1,\, L^i_2\otimes (K_X)_{x_i})\, .
\end{equation}

Take any $$v\, \in\, L^i_1\otimes (K_X)_{x_i}\, \subset\, (J^2({\mathbb T})\otimes K_X)_{x_i}$$
(see \eqref{el}).
Note that ${\mathcal O}_X(-x_i)_{x_i}\,=\, (K_X)_{x_i}$ (see \eqref{ry}). 
Let $\widetilde{v}$ be a holomorphic section of $J^2({\mathbb T})\otimes {\mathcal O}_X(-x_i)$ defined on
some open neighborhood $U$ of $x_i$ such that
\begin{equation}\label{ch1}
\widetilde{v}(x_i)\,=\, v\, ;
\end{equation}
here the identification ${\mathcal O}_X(-x_i)_{x_i}\,=\, (K_X)_{x_i}$ is used. Fix
the open subset $U$ such that $U\bigcap {\mathbb S}\,=\, x_i$.

In particular, $\widetilde{v}$ is a holomorphic section of $J^2({\mathbb T})$ over $U$, and we have
$$
{\mathbf D}(\widetilde{v})\, \in\, H^0(U,\, J^2({\mathbb T})\otimes K_X\otimes{\mathcal O}_X(S))
\,=\, H^0(U,\, J^2({\mathbb T})\otimes K_X\otimes {\mathcal O}_X(x_i))\, .
$$

We will show that ${\mathbf D}(\widetilde{v})$ lies in the image of the natural inclusion map
$$H^0(U,\, J^2({\mathbb T})\otimes K_X\otimes{\mathcal O}_X(-x_i))\, \hookrightarrow\,
H^0(U,\, J^2({\mathbb T})\otimes K_X\otimes{\mathcal O}_X(x_i))\, .$$
For that, first note that the section $\widetilde{v}$ can be expressed as
$$
\widetilde{v}\,=\, f\cdot s_1+ s_2\, ,
$$
where
\begin{itemize}
\item $f$ is a holomorphic function on $U$ with $f(x_i)\,=\, 0$, 

\item $s_1\, \in\, H^0(U,\, J^2({\mathbb T}))$ with $s_1(x_i)\, \in\, L^i_1$ (see \eqref{el}), and

\item $s_2\, \in\, H^0(U,\, J^2({\mathbb T}))$, and it vanishes at $x_i$ of order at least two.
\end{itemize}
Now consider the section
$$
{\mathbf D}(\widetilde{v})\,=\, {\mathbf D}(fs_1)+{\mathbf D}(s_2)\, .
$$
Since $s_2$ vanishes at $x_i$ of order at least two, it follows that
$${\mathbf D}(s_2)\,\in\, H^0(U,\, J^2({\mathbb T})\otimes K_X\otimes{\mathcal O}_X(-x_i))\, .$$
Consequently, to prove that
\begin{equation}\label{h1}
{\mathbf D}(\widetilde{v})\,\in \, H^0(U,\, J^2({\mathbb T})\otimes K_X\otimes{\mathcal O}_X(-x_i))\,
\hookrightarrow\, H^0(U,\, J^2({\mathbb T})\otimes K_X\otimes{\mathcal O}_X(x_i))\, ,
\end{equation}
it suffices to show that ${\mathbf D}(fs_1)\,\in \,H^0(U,\, J^2({\mathbb T})\otimes K_X
\otimes{\mathcal O}_X(-x_i))$.

The Leibniz rule for ${\mathbf D}$ says that
\begin{equation}\label{h2}
{\mathbf D}(fs_1)\,=\, f{\mathbf D}(s_1)+df\otimes s_1\, .
\end{equation}
Since $s_1(x_i)\, \in\, L^i_1$, and $L^i_1$ is the eigenspace of ${\rm Res}({\mathbf D}, x_i)$
for the eigenvalue $-1$, from \eqref{h2} it follows that $${\mathbf D}(fs_1)\,\in\,
H^0(U,\, J^2({\mathbb T})\otimes K_X\otimes{\mathcal O}_X(-x_i))\, .$$
Indeed, if we consider $f{\mathbf D}(s_1)$ and $df\otimes s_1$ as sections of
$(J^2({\mathbb T})\otimes K_X)\vert_U$, then $f{\mathbf D}(s_1)(x_i)\,=\, -v$ by the residue condition, and
$(df\otimes s_1)(x_i)\,=\, v$ by \eqref{ch1}. These imply that the section
$$
{\mathbf D}(fs_1) \,\in\, H^0(U,\, J^2({\mathbb T})\otimes K_X)
$$
vanishes at $x_i$, making it a section of $(J^2({\mathbb T})\otimes K_X\otimes{\mathcal O}_X(-x_i))\vert_U$.

Since ${\mathbf D}(fs_1)\,\in\,
H^0(U,\, J^2({\mathbb T})\otimes K_X\otimes{\mathcal O}_X(-x_i))$, we conclude
that \eqref{h1} holds.

Using the decomposition in \eqref{el0}, the fiber $(J^2({\mathbb T})\otimes K_X\otimes{\mathcal O}_X(-S))_{x_i}$
decomposes as
$$
(J^2({\mathbb T})\otimes K_X\otimes{\mathcal O}_X(-S))_{x_i}
$$
\begin{equation}\label{s2}
\,=\, ((L^i_2\otimes K_X\otimes{\mathcal O}_X(-S))_{x_i})\oplus
((L^i_1\otimes K_X\otimes{\mathcal O}_X(-S))_{x_i})\oplus ((L^i_0\otimes K_X\otimes{\mathcal O}_X(-S))_{x_i})\, .
\end{equation}
For the section ${\mathbf D}(\widetilde{v})$ in \eqref{h1}, let
\begin{equation}\label{s3}
{\mathbf D}(\widetilde{v})^i_2\, \in\, (L^i_2\otimes K_X\otimes{\mathcal O}_X(-S))_{x_i}\,=\, L^i_2\otimes
(K^{\otimes 2}_X)_{x_i}
\end{equation}
be the component of ${\mathbf D}(\widetilde{v})(x_i)$
for the decomposition in \eqref{s2}; recall from \eqref{ry} that
we have ${\mathcal O}_X(-S)_{x_i}\,=\, (K_X)_{x_i}$.

We will now show that the element ${\mathbf D}(\widetilde{v})^i_2\, \in\, L^i_2\otimes
(K^{\otimes 2}_X)_{x_i}$ in \eqref{s3} is
independent of the choice of the section $\widetilde{v}$ in \eqref{ch1} that extends $v$.
To prove this, take any
$$
\widehat{v}\, \in\, H^0(U,\, J^2({\mathbb T})\otimes {\mathcal O}_X(-x_i))
$$
such that $\widehat{v}(x_i)\,=\, v$. So the section $\widetilde{v}-\widehat{v}$ of
$J^2({\mathbb T})\vert_U$ vanishes at $x_i$ of order at least two. Therefore, we can write
$$
\widetilde{v}-\widehat{v}\,=\, f_2s_2+ f_1s_1+f_0s_0\, ,
$$
where $f_0,\, f_1,\, f_2$ are holomorphic functions on $U$ vanishing at $x_i$ of order at least two,
and $s_2(x_i)\, \in\, L^i_2$, $s_1(x_i)\, \in\, L^i_1$, $s_0(x_i)\, \in\, L^i_0$.

It is straightforward
to check that ${\mathbf D}(f_1s_1)$ and ${\mathbf D}(f_0s_0)$ do not contribute to the component
$L^i_2\otimes (K^{\otimes 2}_X)_{x_i}$ in \eqref{s2}; as before, ${\mathcal O}_X(-S)_{x_i}$
is identified with $(K_X)_{x_i}$. Therefore, to prove that
${\mathbf D}(\widetilde{v})^i_2$ in \eqref{s3} is
independent of the choice of the section $\widetilde{v}$, it suffices to show that
${\mathbf D}(f_2s_2)$ also does not contribute to the component
$L^i_2\otimes (K^{\otimes 2}_X)_{x_i}$ in \eqref{s2}. But this follows from the facts that
$s_2(x_i)\, \in\, L^i_2$, and $L^i_2$ is the eigenspace of 
${\rm Res}({\mathbf D}, x_i)$ for the eigenvalue $-2$. Hence we conclude that ${\mathbf D}
(\widetilde{v})^i_2$ is independent of the choice of $\widetilde{v}$.

Now we construct the homomorphism in $\varphi_i$ in \eqref{el2} by sending any $v\, \in\, 
L^i_1\otimes (K_X)_{x_i}$ (as in \eqref{ch1}) to ${\mathbf D}(\widetilde{v})^i_2$ in \eqref{s3}
constructed from $v$.

\begin{theorem}\label{thm2}
The pair $({\mathbb B},\, {\mathbf D})$ in \eqref{tp} defines a
branched ${\rm SO}(3,{\mathbb C})$--oper if and only if 
$\varphi_i\,=\, 0$ for every $x_i\, \in\, \mathbb S$, where $\varphi_i$ is the homomorphism
in \eqref{el2}.
\end{theorem}

\begin{proof}
We first invoke the algorithm, described at the end of Section \ref{se6.4}, to recover
a branched $\text{SO}(3,{\mathbb C})$--oper from the corresponding logarithmic connection and
the twisted bilinear form on $J^2({\mathbb T})$.

Let $\mathcal F$ be the
holomorphic vector bundle on $X$ that fits in the short exact sequence of coherent analytic sheaves on $X$
\begin{equation}\label{t1}
0\, \longrightarrow\, {\mathcal F} \, \longrightarrow\, J^2({\mathbb T})\, \longrightarrow\,
\bigoplus_{i=1}^d J^2({\mathbb T})_{x_i}/L^i_0 \, \longrightarrow\, 0\, ,
\end{equation}
where $L^i_0$ is the eigenspace in \eqref{el}.
Applying the criterion for a logarithmic connection on the vector
bundle $\mathcal W$ in \eqref{ed} to induce a logarithmic connection on the vector
bundle $V$ in \eqref{ed}, the logarithmic connection ${\mathbf D}$
on $J^2({\mathbb T})$ induces a logarithmic connection
${\mathbf D}'$ on $\mathcal F$; the eigenvalues of the residue ${\rm Res}({\mathbf D}', x_i)$
of ${\mathbf D}'$ at any $x_i\, \in\, {\mathbb S}$ are $(-1,\, 0,\, 0)$.
For each point $x_i\, \in \, \mathbb S$, let $M_i\, \subset\, {\mathcal F}_{x_i}$ be the
eigenspace for the eigenvalue $0$ of the residue ${\rm Res}({\mathbf D}', x_i)$. Let $\mathcal E$ be the
holomorphic vector bundle on $X$ that fits in the short exact sequence
\begin{equation}\label{t2}
0\, \longrightarrow\, {\mathcal E} \, \longrightarrow\, {\mathcal F}\, \longrightarrow\,
\bigoplus_{i=1}^d {\mathcal F}_{x_i}/M_i \, \longrightarrow\, 0
\end{equation}
of coherent analytic sheaves on $X$. Applying the above mentioned criterion
we conclude that the logarithmic connection ${\mathbf D}'$ induces a logarithmic connection
${\mathbf D}''$ on $\mathcal E$; the eigenvalues of the residue ${\rm Res}({\mathbf D}'', x_i)$
of ${\mathbf D}''$ at $x_i\, \in\, {\mathbb S}$ are $(0,\, 0,\, 0)$. 

Although all the eigenvalues of ${\rm Res}({\mathbf D}'', x_i)$ are zero, the residue
${\rm Res}({\mathbf D}'', x_i)$ need not vanish in general; it can be
a nilpotent endomorphism. We shall investigate the
residue ${\rm Res}({\mathbf D}'', x_i)$.

Let
$$
\iota\,:\, {\mathcal E} \hookrightarrow\, J^2({\mathbb T})
$$
be the inclusion map obtained from the injective homomorphisms in \eqref{t1} and \eqref{t2}.
Since $\iota$ is an isomorphism over ${\mathbb X}\,=\, X\setminus\mathbb S$,
any holomorphic subbundle $V$ of $J^2({\mathbb T})$ generates a holomorphic subbundle $\widetilde{V}$
of ${\mathcal E}$. This $\widetilde{V}$ is uniquely determined by the condition that 
the space of holomorphic sections
of $\widetilde{V}$ over any open subset $U\, \subset\, X$ is the space of
holomorphic sections $s$ of ${\mathcal E}\vert_U$ such that $s\vert_{U\cap (X\setminus{\mathbb S})}$
is a section of $V$ over $U\cap (X\setminus{\mathbb S})$.

Let $\widetilde{F}^1_{\mathbb T}$ and $\widetilde{F}^2_{\mathbb T}$ be the holomorphic subbundles
of ${\mathcal E}$ corresponding to the holomorphic subbundles $F^1_{\mathbb T}$ and
$F^2_{\mathbb T}$ of $J^2({\mathbb T})$ in \eqref{e11}.

It is straightforward to check that for any point $x\, \in\, \mathbb S$,
\begin{equation}\label{z1}
{\rm Res}({\mathbf D}'', x_i)((\widetilde{F}^1_{\mathbb T})_{x_i})\,=\, 0\ \
\text{ and }\ \ {\rm Res}({\mathbf D}'', x_i)((\widetilde{F}^2_{\mathbb T})_{x_i})\, \subseteq\,
(\widetilde{F}^1_{\mathbb T})_{x_i}\, .
\end{equation}
Moreover, the homomorphism $(\widetilde{F}^2_{\mathbb T})_{x_i}/(\widetilde{F}^1_{\mathbb T})_{x_i}
\,\longrightarrow\, (\widetilde{F}^1_{\mathbb T})_{x_i}$ induced by ${\rm Res}({\mathbf D}'', x_i)$
coincides with the homomorphism $\varphi_i$ in \eqref{el2}.

If $({\mathbb B},\, {\mathbf D})$ in \eqref{tp} defines a
branched ${\rm SO}(3,{\mathbb C})$--oper, then ${\mathbf D}''$ is a nonsingular connection, meaning
${\rm Res}({\mathbf D}'', x_i)\,=\, 0$ for every $x_i\, \in\, \mathbb S$, and consequently,
$\varphi_i\,=\, 0$ for all $x_i\, \in\, \mathbb S$.

Conversely, if $\varphi_i\,=\, 0$ for all $x_i\, \in\, \mathbb S$, then using \eqref{z1} it follows
that ${\rm Res}({\mathbf D}'', x_i)((\widetilde{F}^2_{\mathbb T})_{x_i})\, =\,0$. Now from
\eqref{b1} we conclude that ${\rm Res}({\mathbf D}'', x_i)\,=\, 0$. Hence the
logarithmic connection ${\mathbf D}''$ is nonsingular. Therefore, $({\mathbb B},\,
{\mathbf D})$ defines a branched ${\rm SO}(3,{\mathbb C})$--oper. This completes the proof.
\end{proof}

\begin{remark}\label{reml}
Assume that the local monodromy of the logarithmic connection ${\mathbf D}$ around every point of
$\mathbb S$ is trivial. Then the local monodromy of the logarithmic connection
${\mathbf D}''$ around every point of $\mathbb S$ is trivial, because the monodromies of
${\mathbf D}$ and ${\mathbf D}''$ coincide. Consider the following four facts:
\begin{enumerate}
\item For every point $x_i\, \in\, \mathbb S$, the eigenvalues of ${\rm Res}({\mathbf D}'', x_i)$
are $(0,\, 0,\, 0)$.

\item The local monodromy of ${\mathbf D}''$ around every $x_i\, \in\, \mathbb S$ is trivial.

\item The local monodromy, around $x_i\, \in\, \mathbb S$, 
of any logarithmic connection $D^0$ lies in the conjugacy class of $\exp (-2\pi\sqrt{-1}{\rm Res}(D^0, x_i))$.

\item $\exp (-2\pi\sqrt{-1} A)\, \not=\, I$ for any nonzero nilpotent complex matrix $A$.
\end{enumerate}
These together imply that ${\rm Res}({\mathbf D}'', x_i)\,=\, 0$ for every point $x_i\, \in\, \mathbb S$.
Hence ${\mathbf D}''$ is a nonsingular connection, and $({\mathbb B},\, {\mathbf D})$ in \eqref{tp} defines a
branched ${\rm SO}(3,{\mathbb C})$--oper.
\end{remark}

\section*{Acknowledgements}

This work has been supported by the French government through the UCAJEDI Investments in the 
Future project managed by the National Research Agency (ANR) with the reference number 
ANR2152IDEX201. The first-named author is partially supported by a J. C. Bose Fellowship, and 
school of mathematics, TIFR, is supported by 12-R$\&$D-TFR-5.01-0500. 


\end{document}